\documentclass{amsart}
\usepackage[english]{babel}
\usepackage[latin1]{inputenc}
\usepackage{amssymb}
\usepackage{amsmath}
\usepackage{amsfonts}
\usepackage{amsthm}
\usepackage{enumerate}
\usepackage{color}
\newtheorem{hypothesis}{Hypothesis}
\newtheorem{theorem}{Theorem}[section]
\newtheorem{lemma}[theorem]{Lemma}
\newtheorem{proposition}[theorem]{Proposition}
\newtheorem{corollary}{Corollary}[section]

\theoremstyle{definition}
\newtheorem{definition}[theorem]{Definition}

\theoremstyle{remark}
\newtheorem{remark}[theorem]{Remark}
\def\R{{\mathbb R}}

\newcommand{\N}{\mathbb{N}}

\def\F{\mathcal{F}}
\numberwithin{equation}{section}

\newcommand\cro[1]{\langle #1 \rangle}

\begin{document}

\title[Well-posedness for dispersive Burgers equations]{On well-posedness for some dispersive perturbations of  Burgers' equation}
\author[L. Molinet, D. Pilod and S. Vento]{Luc Molinet, Didier Pilod  and St\'ephane Vento}

\address{Luc Molinet,  Institut Denis Poisson, Universit\'e de Tours, Universit\'e d'Orl\'eans, CNRS (UMR 7013), Parc Grandmont, 37200 Tours, France.}
\email{Luc.Molinet@univ-tours.fr}

\address{Didier Pilod\footnote{Current address: Department of Mathematics, University of Bergen, PO Box 7803, 5020 Bergen, Norway}, Instituto de Matem\'atica, Universidade Federal do Rio de Janeiro, Caixa Postal 68530, CEP: 21945-970, Rio de Janeiro, RJ, Brasil. }
\email{didierpilod@gmail.com}

\address{St\'ephane Vento, Universit\'e Paris 13, Sorbonne Paris Cit\'e, LAGA, CNRS ( UMR 7539),  99, avenue Jean-Baptiste Cl\'ement, F-93 430 Villetaneuse, France.}
\email{vento@math.univ-paris13.fr}

\begin{abstract}
We show that the Cauchy problem for a class of dispersive perturbations of Burgers' equations containing the low dispersion Benjamin-Ono equation
\begin{displaymath}
\partial_tu-D^{\alpha}_x\partial_xu=\partial_x(u^2) \, ,   \quad 0<\alpha\le 1 ,
\end{displaymath}
 is locally well-posed in $H^s(\R)$ when $s>s_\alpha: = \frac 32-\frac {5\alpha} 4$. As a consequence, we obtain global well-posedness in the energy space $H^{\frac{\alpha}2}(\R)$ as soon as $\frac\alpha 2> s_\alpha$, i.e. $\alpha>\frac67$.
\end{abstract}

\date{\today}
\maketitle

\vspace{-0.5cm}

\section{Introduction}
\bigskip

This paper is concerned with the initial value problem for a class of dispersive perturbations of Burgers' equation containing in particular the low dispersion Benjamin-Ono equation
\begin{equation} \label{fKdV}
\partial_tu-D^{\alpha}_x\partial_xu=\partial_x(u^2) \, ,
\end{equation}
where $u=u(x,t)$ is a real valued function, $x \in \mathbb R$, $t \in \mathbb R$, $\alpha >0$  and $D^{\alpha}_x$ is the Riesz potential of order $-\alpha$, which is given via Fourier transform by
\begin{displaymath}
\widehat{D_x^{\alpha}\phi}(\xi)=|\xi|^{\alpha}\widehat{\phi}(\xi) \, .
\end{displaymath}
The cases $\alpha=2$ and $\alpha=1$ correspond to the well-known Korteweg-de Vries (KdV) and Benjamin-Ono (BO) equations. In the case $\alpha=0$, $\partial_xu$ is a transport term, so that there is no dispersion anymore and equation \eqref{fKdV} corresponds merely to the inviscid Burgers equation.

\smallskip
While the Cauchy problem associated with \eqref{fKdV} is now very well-understood in the case $\alpha \ge 1$, our objective here is to investigate the case of low dispersion when $0<\alpha<1$, which seems of great physical interest (see for example the introductions in \cite{KlSa,LiPiSa1} and the references therein).
In particular, in the case $\alpha=\frac12$, the dispersion is somehow reminiscent  of the linear dispersion of finite depth water waves with surface tension. The corresponding Whitham equation with surface tension writes
\begin{equation} \label{whitham}
\partial_tu-w(D_x)\partial_xu +\partial_x(u^2) = 0 \, ,
\end{equation}
where $u=u(x,t)$ is a real valued function, $x \in \mathbb R$, $t \in \mathbb R$, $\omega(D_x)$ is  the Fourier multiplier of symbol $w(\xi)=\left( \frac{\tanh (\xi)}{\xi} \right)^{\frac12}\left(1+\tau\xi^2\right)^{\frac12}$ and $\tau$ is a positive parameter related to the surface tension. Note that for high frequencies $w(\xi) \sim |\xi|^{\frac12}$,  which corresponds exactly to equation \eqref{fKdV} in the $L^2$ critical case.

\smallskip
Equation \eqref{fKdV} is hamiltonian. In particular, the quantities
\begin{equation} \label{M}
M(u)=\int_{\mathbb R}u^2(x,t) \, dx \, ,
\end{equation}
and
\begin{equation} \label{H}
H(u)=\int_{\mathbb R}\big( \frac{1}{2} |D^{\frac{\alpha}2}u(x,t)|^2+\frac{1}{3}u^3(x,t)\big) \, dx \, .
\end{equation}
are (at least formally) conserved by the flow associated to \eqref{fKdV}.
Moreover, equation \eqref{fKdV} is invariant under the scaling transformation
\begin{displaymath}
u_{\lambda} (x,t)=\lambda^{\alpha}u(\lambda x,\lambda^{\alpha+1}t) \, ,
\end{displaymath}
for any positive number $\lambda$. A straightforward computation shows that $\|u_{\lambda}\|_{\dot{H}^s}=\lambda^{s+\alpha-\frac{1}{2}}\|u_{\lambda}\|_{\dot{H}^s}$, and thus the critical index corresponding to \eqref{fKdV} is $\tilde{s}_{\alpha}=\frac{1}{2}-\alpha$. In particular, equation \eqref{fKdV} is $L^2$-critical for $\alpha=\frac{1}{2}$ and \emph{energy} critical for $\alpha=\frac13$.

\medskip
Next we recall some important facts about the initial value problem (IVP) associated with \eqref{fKdV} in $L^2$-based Sobolev spaces $H^s(\mathbb R)$\footnote{ Recall that the natural space where the quantities \eqref{M} and \eqref{H} make sense is $H^{\frac{\alpha}{2}}(\mathbb R)$, at least when $\alpha \ge \frac13$.}. For results in weighted Sobolev spaces, we refer to Fonseca, Linares and Ponce \cite{FoLiPo} and the references therein. It was proved by Molinet, Saut and Tzvetkov \cite{MoSaTz}, that, due to bad high-low frequency interactions in the nonlinearity, the IVP associated with \eqref{fKdV} cannot be solved by a contraction argument on the corresponding integral equation in any Sobolev space $H^s(\mathbb R)$, $s \in \mathbb R$, as soon as $\alpha< 2$. Thus, one needs to use compactness arguments based on \textit{a priori} estimates on the solution and on the difference of two solutions at the required level of regularity.

\smallskip
Standard energy estimates, the Kato-Ponce commutator estimate and Gronwall's inequality provide the following bound for solutions of \eqref{fKdV}
\begin{displaymath}
\|u\|_{L^{\infty}_TH^s_x} \le c \|u_0\|_{H^s_x}e^{c\int_0^T\|\partial_xu\|_{L^{\infty}_x}dt}.
\end{displaymath}
Therefore, one way to obtain \textit{a priori} estimates in $H^s$ is to control $\|\partial_xu\|_{L^1_TL^{\infty}_x}$ at the $H^s$-level. This can be done easily  in $H^{\frac32+}(\mathbb R)$ by using the Sobolev embedding $H^{\frac12+}(\mathbb R) \hookrightarrow L^{\infty}(\mathbb R)$. In the Bejamin-Ono  case $\alpha=1$, Ponce \cite{Po} used the smoothing effects (Strichartz estimates, Kato type smoothing estimate and maximal function estimate) associated with the dispersive part of \eqref{fKdV} to obtain well-posedness in $H^{\frac32}(\mathbb R)$. Later on, Koch and Tzvetkov \cite{KoTz} introduced a refined Strichartz estimate, derived by chopping the time interval in small pieces whose length depends on the spatial frequency of the solution, which allowed them to prove local well-posedness for BO in $H^{\frac54+}(\mathbb R)$. This refined Strichartz estimate was then improved by Kenig and Koenig \cite{KK} and the local well-posedness for BO pushed down to $H^{\frac98+}(\mathbb R)$. Recently, Linares, Pilod and Saut \cite{LiPiSa1} extended Kenig and Koenig's result to \eqref{fKdV} in the range $0<\alpha<1$ by proving that the corresponding initial value problem is well-posed in $H^s(\mathbb R)$ for $s>\frac32-\frac{3\alpha}8$. Note that even very few dispersion (when $0<\alpha \ll 1$) allows to enlarge the resolution space, which is not the case anymore when there is no dispersion. Indeed, it is known that the IVP associated with Burgers' equation is ill-posed in $H^{\frac32}(\mathbb R)$ (c.f. Remark 1.6. in \cite{LiPiSa1}).

\smallskip
Another technique to obtain suitable estimates on the solutions at low regula-rity is the use of a nonlinear \emph{gauge} transformation which allows to weaken the bad frequency interactions in the nonlinear term. Such transformation was introduced by Tao \cite{Ta} for the Benjamin-Ono equation and enabled him to prove global well-posedness for BO in $H^1(\mathbb R)$. By using this \emph{gauge} transformation in the context of Bourgain's spaces $X^{s,b}$, Burq and Planchon \cite{BuPl}, respectively Ionescu and Kenig \cite{IK}, proved that the IVP associated with BO is well-posed in $H^{\frac14+}(\mathbb R)$, respectively $L^2(\mathbb R)$. We also refer to Molinet and Pilod \cite{MoPi} for another proof of Ionescu and Kenig's result with stronger uniqueness result (for example unconditional uniqueness in $H^{\frac14+}(\mathbb R)$). In  \cite{HIKK}, Herr, Ionescu, Kenig and Koch were able to extend Ionescu and Kenig's result  to the whole range $1 < \alpha <2$.  By using a paradifferential gauge transformation, they proved that the IVP associated to \eqref{fKdV} is globally well-posed in $L^2(\mathbb R)$ for $1 < \alpha <2$.

\smallskip
Recently Molinet and Vento \cite{MV} introduced a new method to obtain energy estimates at low regularity for
strongly nonresonant dispersive equations. It starts with the classical  estimate for the dyadic piece $P_Nu$  localized in turn of the spatial frequency $N$,
\begin{equation} \label{energy}
\|P_Nu\|_{L^{\infty}_TL^2_x}^2 \lesssim \|P_Nu_0\|_{L^2_x}^2+\sup_{t \in ]0,T[}\left| \int_0^t\int_{\mathbb R}P_N\partial_x(u^2)P_Nu dxdt \right| \, .
\end{equation}
To control the last term on the right-hand side of the energy estimate \eqref{energy}, one performs a paraproduct decomposition
\begin{equation} \label{energy2}
\int_{\mathbb R \times [0,t]}P_N\partial_x(u^2)P_Nu  =\int_{\mathbb R \times [0,t]}\partial_xP_N(u_{\gtrsim N}u_{\gtrsim N})P_Nu +
\int_{\mathbb R \times [0,t]}\partial_xP_N(u_{\ll N}u)P_Nu  \,
\end{equation}
and put the derivative on the lowest spatial frequencies by \lq\lq integrating by parts\rq\rq\footnote{Since we work with frequency localized functions, this corresponds actually to use suitable commutator estimates. }.
The idea is then to perform a dyadic decomposition of each function  in term of its modulation variable and to put one of them (the one with the greatest modulation) in the space $X^{s-1,1}$. This allows to recover at least $|\Omega| N^{-1} $  where $ \Omega$ is the resonance function.  The price to pay is to handle the characteristic function $1_{]0,t[}$ which appears after extending the functions to $\mathbb R^2$ and is not continuous in $X^{s-1,1}$. On the positive side, the $X^{s-1,1}$  norm of $u$ is relatively simple to control by using the classical linear estimates in Bourgain's spaces as follows
\begin{equation} \label{bilinear}
\|u\|_{X^{s-1,1}} \lesssim \|u_0\|_{H^{s-1}}+\|\partial_xJ^{s-1}_x(u^2)\|_{L^2_{x,T}} \lesssim \|u_0\|_{H^{s}}+\|J^s_x(u^2)\|_{L^{\infty}_TH^s_x} \, .
\end{equation}
Thus, for $s>\frac12$, one can easily concludes the bilinear estimate since $H^s(\mathbb R)$ is a Banach algebra. By using this method, Molinet and Vento proved that the IVP associated with \eqref{fKdV} is locally well-posed in $H^s(\mathbb R)$ for $s \ge 1-\frac{\alpha}2$ when $1 \le \alpha \le 2$.
Note that Guo  \cite{Guo} also proved local well-posedness in $H^s(\mathbb R)$ for $s>2-\alpha$ when $1  \le \alpha \le 2$ without using a gauge transformation.  He used instead the short time Bourgain's spaces in the way of Ionescu, Kenig and Tataru in \cite{IKT}.

\medskip
Throughout this paper we consider the class of dispersive equations
\begin{equation}\label{dpB}
  \partial_t u+ L_{\alpha+1}u = \partial_x(u^2),
\end{equation}
where $u=u(x,t)$ is a real-valued function, $x\in\R$, $t\in\R$, $\alpha>0$ and the linear operator $L_{\alpha+1}$ satisfies the following hypothesis.
\begin{hypothesis}	\label{hyp1}
We assume that $L_{\alpha+1}$ is the Fourier multiplier operator by $i\omega_{\alpha+1}$ where $\omega_{\alpha+1}$ is a real-valued odd function belonging to $C^1(\R)\cap C^\infty(\R^\ast)$ and satisfying:  There exists $\xi_0>0$ such that for any $\xi\ge \xi_0$, it holds
\begin{equation}\label{hyp1.1}
 |\partial^\beta \omega_{\alpha+1}(\xi)| \sim |\xi|^{\alpha+1-\beta},\ \beta\in \{0,1,2\},
\end{equation}
and
\begin{equation}\label{hyp1.2}
  |\partial^\beta \omega_{\alpha+1}(\xi)| \lesssim |\xi|^{\alpha+1-\beta},\ \beta\ge 3.
\end{equation}
\end{hypothesis}

\begin{remark}\label{exemplesOp}
  We easily check that the following operators satisfy Hypothesis \ref{hyp1}:
  \begin{enumerate}
    \item The purely dispersive operator $L_{\alpha+1}=-D_x^\alpha \partial_x$, $\alpha>0$.
    \item The Whitham operator with symbol $\omega(\xi)=\xi\left( \frac{\tanh (\xi)}{\xi} \right)^{\frac12}\left(1+\tau\xi^2\right)^{\frac12}$, $\tau>0$ for $\alpha=1/2$.
    \item The linear Intermediate Long Wave operator $L_{\alpha+1}=\partial_x D_x \coth(D_x)$ for $\alpha=1$.
  \end{enumerate}
\end{remark}

In this article, we show that the initial value problem (IVP) associated with \eqref{dpB} is locally well-posed in $H^s(\mathbb R)$ for $s>\frac32-\frac{5\alpha}4$ when $0<\alpha \le 1$, which improves Linares, Pilod and Saut's result in \cite{LiPiSa1}.

\begin{theorem} \label{maintheo}
Assume that $L_{\alpha+1}$ satisfies Hypothesis \ref{hyp1} with $0<\alpha \le 1$ and let $s>s_\alpha = \frac 32- \frac{5\alpha} 4$. Then, for any $u_0 \in H^s(\mathbb R)$, there exist $T=T(\|u_0\|_{H^s})>0$ and a unique solution $u$ of the IVP associated with \eqref{dpB} in the class
\begin{equation} \label{maintheo.100}
C([0,T] : H^s(\mathbb R)) \cap X^{s-1,1}_T \cap L^2(0,T : W^{s-s_{\alpha}+(1-\alpha)_-,\infty}(\mathbb R)) \, .
\end{equation}

Moreover, for any $0<T'<T$, there exists a neighborhood $\mathcal{U}$ of $u_0$ in $H^s(\mathbb R)$ such that the flow-map data solution
 $ v_0\mapsto v $ is continuous from $  \mathcal{U} $ into $  C([0,T'] : H^s(\mathbb R))$.
\end{theorem}

\begin{remark} In the case $\alpha=1$ and $L_{\alpha+1}=-D_x^\alpha\partial_x$, our result provides a proof of the local well-posedness for BO in $H^{\frac14+}(\mathbb R)$. In other words, we recover Burq and Planchon's result in \cite{BuPl} without using a gauge transformation.
\end{remark}

If we assume moreover that the symbol $\omega_{\alpha+1}$ satisfies
\begin{equation}\label{hyp2}
  |\omega_{\alpha+1}(\xi)|\lesssim |\xi| \text{ for } |\xi| \lesssim 1,
\end{equation}
we easily see that the Hamiltonian
$$
H_{\alpha+1}(u) = \int_\R (\frac 12 |\Lambda^{\alpha/2}u(x,t)|^2+\frac 13 u^3(x,t))dx
$$
where $\Lambda^{\alpha/2}$ is the space Fourier multiplier defined by
$$
\widehat{\Lambda^{\alpha/2}v}(\xi) = \left|\frac{\omega_{\alpha+1}(\xi)}{\xi} \right|^{1/2} \hat{v}(\xi),
$$
as well as \eqref{M} are conserved by the flow associated to \eqref{dpB}. Iterating Theorem \ref{maintheo}, we obtain global well-posedness as soon as $\alpha>\frac67$.

\begin{corollary}
Assume that $L_{\alpha+1}$ satisfies Hypothesis \ref{hyp1} and \eqref{hyp2} with $\frac 67 < \alpha \le 1$. Then the Cauchy problem associated with \eqref{dpB} is globally well-posed in the energy space $H^{\frac\alpha2}(\R)$.
\end{corollary}

\begin{remark}
  The operators defined in Remark \ref{exemplesOp} also satisfy assumption \eqref{hyp2}.
\end{remark}

\begin{remark}
Based on numerical computations by Klein and Saut \cite{KlSa}, the global well-posedness  of \eqref{fKdV} was conjectured \cite{KlSa,LiPiSa1} in the $L^2$-subcritical case $\alpha>\frac12$. Here, we answer to part of this conjecture when $\alpha>\frac67$. Up to our knowledge, this is the first global existence result for $ \alpha<1$.
\end{remark}

\begin{remark}
It would be interesting to obtain results on the dispersion decay of the solutions associated to small data for \eqref{fKdV} with low dispersion. Some progress in this direction were recently done by Ifrim and Tataru \cite{IfTa} for the Benjamin-Ono equation.\footnote{Note also that the authors give another proof of the well-posedness of the Benjamin-Ono equation in $L^2$ without using the $X^{s,b}$ structure but still based on Tao's renormalization argument together with modified energies.}
\end{remark}

\begin{remark}
In \cite{LiPiSa2}, Linares, Pilod and Saut showed that the solitary waves associated to \eqref{fKdV} are orbitally stable in the energy space $H^{\frac{\alpha}2}(\mathbb R)$ as soon as $\alpha>\frac12$, conditionally to the global well-posedness in $H^{\frac{\alpha}2}$ (see Remark 2.1 in \cite{LiPiSa2}). We also refer to Arnesen \cite{Ar} and Angulo \cite{An} for other proofs of this result. Theorem 2.14 in \cite{LiPiSa2} combined with Theorem \ref{maintheo} provides then a complete orbital stability result in the energy space as soon as $\alpha>\frac67$.
\end{remark}

\medskip

Now, we discuss the main ingredients in the proof of Theorem \ref{maintheo}. Since it is not clear wether one can take advantage of a gauge transformation in the case $\alpha<1$ or not, we elect to follow the energy method introduced in \cite{MV}. However, we need to add several key ingredients.

\smallskip
Firstly, in order to close the bilinear estimate \eqref{bilinear} in $H^s(\mathbb R)$ for $s \le \frac12$, we use the  norm $\|\cdot\|_{L^2_TL^{\infty}_x}$, which is in turn estimated by using the refined Strichartz estimate as in \cite{KoTz,KK,LiPiSa1}. Then, we can control the last term on the right-hand side of \eqref{bilinear} by using the fractional Leibniz rule as $\|J^s_x(u^2)\|_{L^{\infty}_TH^s_x} \lesssim \|u\|_{L^2_TL^{\infty}_x}\|J^s_xu\|_{L^{\infty}_TL^2_x}$.

\smallskip
The norm $\|\cdot\|_{L^2_TL^{\infty}_x}$ is also an important ingredient to close the energy estimate \eqref{energy}. This creates a serious technical difficulty. Indeed to handle some commutators with those norms, we need then to use a  generalized Coifman-Meyer theorem for multilinear Fourier multipliers $m(\xi_1,\cdots,\xi_n)$ satisfying the Marcinkiewicz type condition
$$|\partial^{\beta}m (\xi_1,\cdots,\xi_n)| \lesssim \prod_{i=1}^n  |\xi_i|^{-\beta_i},\quad \forall \, \beta \in \mathbb N^n \, .$$
Such a theorem was proved by Muscalu, Pipher, Tao and Thiele \cite{MuPiTaoThi} in the bilinear case and can be deduced from a result of Bernicot \cite{B} in the multilinear case (see Section 2.3 for more details).

\smallskip
With this theorem in hand, we can estimate the first term of \eqref{energy2}
  corresponding to the \emph{high-high} frequency interactions by using the norms $\|u\|_{X^{s-1,1}_T}$ and $\|J_x^{1-\alpha}u\|_{L^2_TL^{\infty}_x}$ as explained above. For the second term, we would like to integrate by parts and use the $ \|\cdot \|_{X^{s-1,1}}$-norm as in \cite{MV} but  the resonance relation $|\Omega| \sim N_{min}N_{max}^{\alpha}$ would not be sufficient to recover the ``big'' derivative we lost by using this norm. This is one of the main difficulty to work at low dispersion $\alpha<1$. For this reason, we modify the energy by adding a cubic term, constructed so that the contribution of its time derivative coming from the linear part of the equation cancels out the \emph{high-low} frequency term. It is worth noticing that this modified energy is defined in Fourier variables in the same spirit of the modified energy in the I-method \cite{CKSTT}. We also refer to our recent works \cite{MoPiVe1,MoPiVe2} on the modified Korteweg-de Vries equation both on the line and on the torus for a similar strategy using a modified energy. Note that we gain a factor
  $ N_{min}N_{max}^{\alpha} $ on the additional cubic term. On the other hand, the contribution of its time derivative  coming from the nonlinear part of the equation   is of order four and contains one more spatial derivative. For $ \alpha<1$, it is clear that when this spatial derivative falls on the term with the highest spatial frequencies we should lose $ N_{min}^{-1}N_{max}^{1-\alpha}$ which is not acceptable for some high-low frequency interaction terms. The crucial observation here is that there is
  a fundamental cancellation between two of those terms exhibiting the badest high-low frequency interactions.

\smallskip
Those ingredients are enough to derive a suitable \textit{a priori} estimate for a solution of \eqref{dpB}. However, things are more complicated  to get an estimate for the difference of two solutions $u_1$ and $u_2$, since the corresponding equation lacks of symmetry. For this reason, we are only able to derive an energy estimate for the difference $w=u_1-u_2$ at low regularity $H^{\sigma}$, $\sigma<0$, and with an additional weight on low frequency. This is sufficient for our purposes, since we only need this estimate for the difference of solutions having the same low frequency part in order to prove the uniqueness and the continuity of the flow map (c.f. \cite{IK}).
  However, the bilinear estimate is not straightforward as before when working with negative regularity $H^{\sigma}$, $\sigma<0$. To overcome this last difficulty, we follow the strategy in \cite{MV} and work with the sum space $F^{s,\frac12}=X^{s-1,1}+X^{s,(\frac12)_+}$ instead of working with $X^{s-1,1}$ only.

\smallskip
Finally, it is worth noticing that even in the particular case of purely power dispersion where scaling invariance occurs, equation \eqref{dpB} is $ L^2$-super critical for $ \alpha<1/2$ and thus we will not be able to use a classical scaling argument to prove the local existence result. Roughly speaking, our method consists in cutting  the spatial frequencies of the solution into two parts $ P_{\le N_0} $ and $ P_{>N_0} $.
We gain some positive factor of the time $ T$ (but lose some positive factor of $ N_0$) when estimating the low frequency part  whereas we gain a negative factor of $ N_0 $ when estimating the high frequency part. This will allow us to close our estimates on $ ]0,T[$ for  smooth solution
 to \eqref{fKdV}  by taking $ N_0$ big enough and $ T >0 $ small enough.
Finally, the continuity of the solution as well as the continuity with respect to initial data will be proved by using a kind of uniform decay estimate on the high spatial frequencies of the solution.

\smallskip
The paper is organized as follows: in Section 2, we introduce the notation, define the function spaces and state some important preliminary estimates related the generalized Coifman-Meyer theorem. In Section 3, we derive multilinear estimates at the $L^2$-level. Those estimates will be used in Sections 4 and 5 to prove estimates for the solution and the difference of two solutions of the equation. Finally, we give the proof of Theorem \ref{maintheo} in Section 6.

\subsection*{Acknowledgements}
The authors would like to thank Jean-Claude Saut for  constant encouragements. They are also grateful to Terence Tao for helpful comments on the generalized Coifman-Meyer theorem in Section 2.3. L.M and S.V were partially
supported by the ANR project GEO-DISP. D.P. was partially supported by CNPq/Brasil grant 3035051/2016-7.

\bigskip
\bigskip

\section{Notation, function spaces and preliminary estimates}
\subsection{Notation}\label{notation}
For any positive numbers $a$ and $b$, the notation $a\lesssim b$ means that there exists a positive constant $C$ such that $a\le Cb$, and we denote $a\sim b$ when $a\lesssim b$ and $b\lesssim a$. We also write $a\ll b$ if the estimate $b\lesssim a$ does not hold. If $x\in\R$, $x_+$, respectively $x_-$ will denote a number slightly greater, respectively lesser, than $x$. We also set $\cro{x} = (1+x^2)^{\frac12}$.

For $u=u(x,t)\in \mathcal{S}'(\R^2)$, $\F u=\hat{u}$ will denote its space-time Fourier transform, whereas $\F_xu$, respectively $\F_tu$ will denote its Fourier transform in space, respectively in time. For $s\in\R$, we define the Bessel and Riesz potentials of order $-s$, $J^s_x$ and $D^x_x$, by
$$
J^s_xu = \F^{-1}_x(\cro{\xi}^s\F_xu) \textrm{ and } D^s_xu = \F^{-1}_x(|\xi|^s \F_xu).
$$

Throughout the paper, we fix a smooth cutoff function $\eta$ such that
\begin{equation}\label{eta}
\eta \in C_0^{\infty}(\mathbb R), \quad 0 \le \eta \le 1, \quad
\eta_{|_{[-1,1]}}=1 \quad \mbox{and} \quad  \mbox{supp}(\eta)
\subset [-2,2].
\end{equation}
We set  $ \phi(\xi):=\eta(\xi)-\eta(2\xi)$. Let  $\widetilde{\phi} \in C_0^{\infty}(\mathbb R)$ be such that $\widetilde{\phi}_{|_{\pm [\frac12,2]}} \equiv 1$ and $\text{supp} \, (\widetilde{\phi}) \subset \pm [\frac14,4]$. For $l \in \mathbb Z$, we define
$$
\phi_{2^l}(\xi):=\phi(2^{-l}\xi), \quad \ \widetilde{\phi}_{2^l}(\xi)=\phi_{\sim 2^l}(\xi):=\widetilde{\phi}(2^{-l}\xi) \, ,
$$
and, for $ l\in \N^* $,
$$
\psi_{2^{l}}(\xi,\tau)=\phi_{2^{l}}(\tau-\omega_{\alpha+1}(\xi)).
$$
By convention, we also denote
$$
\psi_{{1}}(\xi,\tau):=\eta(2(\tau-\omega_{\alpha+1}(\xi))).
$$
Any summations over capitalized variables such as $N$ or $L$ are presumed to be dyadic. Unless stated otherwise, we work with homogeneous dyadic decomposition for the space frequency variables and non-homogeneous decompositions for modulation variables, i.e. these variables range over numbers of the form $\{2^k : k\in\mathbb Z\}$ and $\{2^k : k\in\mathbb N\}$ respectively.  Then, we have that
$$
\sum_{N>0}\phi_N(\xi)=1\quad \forall \xi\in \R^*, \quad \mbox{supp} \, (\phi_N) \subset
\{\frac{N}{2}\le |\xi| \le 2N\}, \ N \in \{2^k : k\in \mathbb Z\},
$$
and
$$
\sum_{L\ge 1}\psi_L(\xi,\tau)=1 \quad \forall (\xi,\tau)\in\R^2, \quad  L \in \{2^k : k\in \mathbb N\}.
$$

Let us define the Littlewood-Paley multipliers by
$$
P_Nu=\mathcal{F}^{-1}_x\big(\phi_N\mathcal{F}_xu\big), \quad P_{ \sim N}u=\mathcal{F}^{-1}_x\big(\widetilde{\phi}_N\mathcal{F}_xu\big) \quad
Q_Lu=\mathcal{F}^{-1}\big(\psi_L\mathcal{F}u\big),
$$
 $P_{\ge N}:=\sum_{K \ge N} P_{K}$,  $P_{\gtrsim N}:=\sum_{K \ge N} P_{\sim K}$, $P_{\le N}:=\sum_{K \le N} P_{K}$, $P_{\lesssim N}:=\sum_{K \le N} P_{\sim K}$, $Q_{\ge L}:=\sum_{K \ge L} Q_{K}$ and   $Q_{\le L}:=\sum_{K \le L} Q_{K}$. For the sake of brevity we often write $u_N=P_Nu$, $u_{\le N}=P_{\le N}u$, $\cdots$

\smallskip
Finally, if $N_1$, $N_2$ are two dyadic numbers, we denote $N_1 \vee N_2 =\max\{N_1,N_2\}$ and $N_1 \wedge N_2 =\min\{N_1,N_2\}$.

\subsection{Function spaces}
For $1\le p\le\infty$, $L^p$ denotes the usual Lebesgue space and for $s\in\R$, $H^s$ is the $L^2$-based Sobolev space with norm $\|f\|_{H^s}=\|J_x^s f\|_{L^2}$.
If $B$ is a space of functions on $\R$, $T>0$ and $1\le p\le\infty$, we define the spaces $L^p_TB_x$ and $L^p_tB_x$ by the norms
$$
\|f\|_{L^p_TB_x} = \left\| \|f\|_{B}\right\|_{L^p([0,T])} \ \textrm{ and } \|f\|_{L^p_tB_x} = \left\| \|f\|_B \right\|_{L^p(\R)}.
$$
If $M$ is a normed space of  functions, we will denote $\overline{M}$ its subspace associated with the weighted norm:
$$
\|u\|_{\overline{M}} = \|\F_x^{-1}( \cro{|\xi|^{-1}}\F_xu(\xi))\|_M.
$$

For $s,b\in\R$ we introduce the Bourgain space $X^{s,b}$ associated with the dispersive Burgers' equation as the completion of the Schwartz space $\mathcal{S}(\R^2)$ under the norm
$$
\|u\|_{X^{s,b}} = \|\cro{\xi}^s \cro{\tau-\omega_{\alpha+1}(\xi)}^b \F_{tx}u \|_{L^2}.
$$
We will also work in the sum space $F^{s,b} = X^{s-1,b+\frac12}+X^{s,b_+}$ endowed with the norm
\begin{equation} \label{F}
\|u\|_{F^{s,b}} = \inf \big\{ \|u_1\|_{X^{s-1,b+\frac12}} + \|u_2\|_{X^{s,b_+}} \, : \, u=u_1+u_2\big\}.
\end{equation}
For $s\in\R$, we define our resolution space $Y^s$ by the norm
\begin{equation}\label{Y}
\|u\|_{Y^s} = \|u\|_{L^\infty_t H^s_x} + \|u\|_{X^{s-1,1}} + \|J_x^{(s-s_\alpha)+(1-\alpha)_-}u\|_{L^2_t L^\infty_x}.
\end{equation}
We will also need to consider the space $Z^s$ equipped with the norm
$$
\|u\|_{Z^s} = \|u\|_{L^\infty_t H^s_x} + \|u\|_{F^{s,\frac12}} + \|J^{(s-s_\alpha)+(1-\alpha)_-}u\|_{L^2_t L^\infty_x}.
$$
Finally, we will use restriction in time versions of these spaces. Let $T>0$ be a positive time and $M$ be a normed space of space-time functions. The restriction space $M_T$ will be the space of functions $u : \R\times ]0,T[\to\R$ satisfying
$$
\|u\|_{M_T} = \inf \{ \|\widetilde{u}\|_M \, : \, \widetilde{u} \, : \, \R\times\R\to\R, \, \widetilde{u}|_{\R\times ]0,T[} = u\} <\infty.
$$

\subsection{Generalized Coifman-Meyer theorem}\label{pseudoprodest}

\begin{definition} \label{defpseudo}
For $n\ge 1$ and $\chi$ a bounded measurable function on $\R^n$, we define the multilinear Fourier multiplier operator $\Pi^n_\chi$ on $\mathcal{S}(\R)^n$ by
\begin{equation}
 \Pi^n_\chi(f_1,\ldots, f_n))(x) = \int_{\mathbb R^n} \chi(\xi_1,\ldots,\xi_n)\prod_{j=1}^n\widehat{f}_j(\xi_j)e^{ix(\xi_1+\cdots \xi_n)} \, d\xi_1\cdots d\xi_n \, .
\end{equation}
\end{definition}

If $\sigma$ is a permutation of $\{1,\ldots , n\}$, then it is clear that
\begin{equation}\label{pseudoprodsym.1}
  \Pi_\chi^n(f_1, \ldots, f_n) = \Pi_{\chi_\sigma}^n (f_{\sigma(1)}, \ldots, f_{\sigma(n)})
\end{equation}
where $\chi_\sigma(\xi_1, \ldots, \xi_n) = \chi(\xi_{\sigma(1)}, \ldots, \xi_{\sigma(n)})$. For any $t>0$, we define $\R^n_t = \R^n\times ]0,t[$ and for $u_1,\ldots u_{n+1}\in \mathcal{S}(\R^2)$, we set
\begin{equation}\label{Gtdef}
  G_{t,\chi}^n(u_1, \ldots, u_{n+1}) = \int_{\R_t} \Pi_{\chi}^n(u_1,\ldots, u_n) u_{n+1} \, dxdt \, .
\end{equation}
When there is no risk of confusion, we will write $G_t^n = G_{t,\chi}^n$ with $\chi\in L^\infty(\R^n)$.

From Plancherel theorem, it is not too hard to check that
\begin{equation}\label{Gtsym.1}
  G_{t,\chi }^n(u_1, \ldots, u_{n+1}) = \int_{\R_t} \Pi_{\tilde{\chi}}^n (u_{n+1}, u_2, \ldots, u_{n}) u_1 \, dxdt \,
\end{equation}
where $\tilde{\chi}(\xi_1,\ldots, \xi_n) = \chi(-\sum_{i=1}^n\xi_i, \xi_2, \ldots, \xi_{n})$. We deduce from \eqref{pseudoprodsym.1}-\eqref{Gtsym.1} that
\begin{equation}\label{Gtsym.2}
  G_{t,\chi}^n(u_1, \ldots, u_{n+1}) = G_{t,\chi_{\sigma}}^n(u_{\sigma(1)},\ldots, u_{\sigma(n+1)})
\end{equation}
for any permutation $\sigma$ of $\{1, \ldots, n+1\}$ with an implicit symbol $\chi_{\sigma}\in L^\infty(\R^n)$ satisfying $\|\chi_{\sigma}\|_{L^{\infty}} \lesssim \|\chi\|_{L^{\infty}}$.

\medskip
The classical Coifman-Meyer theorem \cite{CoMe} states that if $\chi$ is smooth away from the origin and satisfies the H\"ormander-Milhin condition
\begin{equation} \label{HM}
|\partial^{\beta}\chi(\xi)| \lesssim |\xi|^{-\beta} \, ,
\end{equation}
for sufficiently many multi-indices $\beta \in \mathbb N^n$, then the operator $\Pi_{\chi}^n$ is bounded from $L^{p_1}(\mathbb R) \times \cdots \times L^{p_n}(\mathbb R)$ to $L^p(\mathbb R)$ and satisfies
\begin{equation} \label{CMtheo}
\|\Pi_{\chi}^n(f_1,\cdots,f_n)\|_{L^p} \lesssim \prod_{j=1}^n\|f_j\|_{L^{p_j}} \, ,
\end{equation}
as long as $1 < p_j \le +\infty$, $1 \le p < +\infty$ and $\frac1p=\frac1{p_1}+\cdots\frac1{p_n}$.

\medskip
In the sequel, we will need the following generalized version of Coifman-Meyer's theorem.

\begin{theorem}\label{pseudoprod}
Let $1 < p_1,\cdots,p_n < +\infty$ and $1 \le p < +\infty$ satisfy $\frac1p=\frac1{p_1}+\cdots\frac1{p_n}$.
Assume that $f_1,\ldots ,f_n\in\mathcal{S}(\R)$ are functions with Fourier variables supported in $\{|\xi|\sim N_i\}$ for some dyadic numbers $N_1,\ldots, N_n$.

Assume also that $\chi\in C^\infty(\R^n)$ satisfies the Marcinkiewicz type condition
\begin{equation}\label{pseudoprod.1}
  \forall \beta=(\beta_1,\ldots, \beta_n)\in \mathbb{N}^n, \quad |\partial^{\beta}\chi (\xi)| \lesssim \prod_{i=1}^n |\xi_i|^{-\beta_i},
\end{equation}
on the support of $\prod_{i=1}^n\hat{f}_i(\xi_i)$. Then,
\begin{equation}\label{pseudoprod.2}
\|\Pi_{\chi}^n(f_1,\cdots,f_n)\|_{L^p} \lesssim \prod_{j=1}^n\|f_j\|_{L^{p_j}} ,
\end{equation}
with an implicit constant that doesn't depend on $N_1,\ldots, N_n$.
\end{theorem}

\begin{remark}
Condition \eqref{HM} is too restrictive for our purpose. For instance if $N_1\ll N_2$ are dyadic numbers and
  $$
  \chi(\xi_1,\xi_2) = \varphi_{N_1}(\xi_1)\varphi_{N_2}(\xi_2),
  $$
  then $\chi$ clearly satisfies condition \eqref{pseudoprod.1}, but $|\partial_{\xi_1}\chi(\xi_1,\xi_2)| \sim N_1^{-1} \gg N_2^{-1} \sim |(\xi_1,\xi_2)|^{-1}$, so that $\chi$ does not satisfy \eqref{HM}.
\end{remark}

\smallskip
Theorem \ref{pseudoprod} was proved by Muscalu, Pipher, Tao and Thiele \cite{MuPiTaoThi} in the case of bilinear Fourier multipliers\footnote{Note that even the extremal case where one the $p_i$ is equal to $+\infty$ is proved.} (in dimension $2$).

One could certainly prove Theorem \ref{pseudoprod} by extending the arguments in \cite{MuPiTaoThi} to the multilinear case\footnote{Personal communication by Terence Tao.}. Instead, we will deduce Theorem \ref{pseudoprod} as a Corollary of Bernicot's theorem in \cite{B}.

\begin{theorem}[\cite{B}, Theorem 1.3]\label{unifEst}
  Suppose $1<p_1,\ldots,p_n<\infty$, $1\le p<\infty$ and $1/p=1/p_1+\ldots +1/p_n$. Assume that $\chi\in C^\infty(\R^n)$ satisfies
  \begin{equation}\label{unifEst.1}
    \forall \beta=(\beta_1,\ldots, \beta_n)\in \mathbb{N}^n,\quad |\partial^{\beta}\chi (\xi)| \lesssim \frac{\prod_{i=1}^n |\lambda_i|^{\beta_i}} {d_\lambda(\xi,0)^{|\beta|}},
  \end{equation}
  for some $\lambda_1,\ldots, \lambda_n >0$ and where $d_\lambda$ is the metric defined by $d_\lambda(\xi,0)=\sum_{i=1}^n \lambda_i|\xi_i|$.
  Then we have for any smooth functions $f_1,\ldots, f_n\in \mathcal{S}(\R)$
  \begin{equation}
    \|\Pi^n_\chi(f_1,\ldots ,f_n)\|_{L^p} \lesssim \prod_{i=1}^n\|f_i\|_{L^{p_i}} ,
  \end{equation}
  with an implicit constant that doesn't depend on $\lambda$.
\end{theorem}

\begin{proof}[Proof of Theorem \ref{pseudoprod}]
  Noticing that
  $$
   \Pi^n_\chi(f_1,\ldots ,f_n) = \Pi^n_{\tilde{\chi}}(f_1,\ldots ,f_n)
  $$
  with $\tilde{\chi}(\xi_1,\cdots,\xi_n) = \chi(\xi_1,\cdots,\xi_n)\prod_{i=1}^n \phi_{\sim N_i}(\xi_i)$, it suffices to show that $\tilde{\chi}$ satisfies \eqref{unifEst.1} for suitable $\lambda_1,\cdots,\lambda_n>0$. But setting $\lambda=(\frac{N_n}{N_1},\ldots ,\frac{N_n}{N_{n-1}}, 1)$, this is easily checked since on the one hand
  $$
  |\partial^\beta \tilde{\chi}(\xi)| \lesssim \sum_{\gamma\le \beta} |\partial^{\beta-\gamma}\chi(\xi)| \prod_{i=1}^n N_i^{-\gamma_i} \tilde{\phi}^{(\gamma_i)}(\frac{\xi_i}{N_i}) \lesssim \prod_{i=1}^n N_i^{-\beta_i} 1_{|\xi_i|\sim N_i},
  $$
  and on the other hand,
  $$
  \frac{\prod_{i=1}^n |\lambda_i|^{\beta_i}} {d_\lambda(\xi,0)^{|\beta|}} \sim \frac{\prod_{i=1}^n \left(\frac{N_n}{N_i}\right)^{\beta_i}} {\left(N_n\sum_{i=1}^n \frac{|\xi_i|}{N_i}\right)^{|\beta|}} \sim \prod_{i=1}^n N_i^{-\beta_i},
  $$
  for $|\xi_i|\sim N_i$.
\end{proof}

\begin{remark}\label{symbolprod}
  It is worth noticing that if two symbols $\chi_1, \chi_2$ satisfy \eqref{pseudoprod.1}, then this condition also holds for the product function $\chi_1\chi_2$. This is easily obtained thanks to the Leibniz rule.
\end{remark}

\begin{lemma}\label{omega.inv}
Let $0<\alpha \le 1$. Let $N_1\ll N_2$ be two dyadic numbers. Then the symbol $\chi$ defined on $\R^2$ by
$$
\chi(\xi_1, \xi_2) = \frac{N_1N_2^\alpha}{\Omega_2(\xi_1, \xi_2)}
$$
where $\Omega_2$ is defined in \eqref{Omega2}, satisfies the Marcinkiewicz condition \eqref{pseudoprod.1} on the set $\{(\xi_1, \xi_2)\in\R^2 : |\xi_1|\sim N_1, |\xi_2|\sim N_2\}$.
\end{lemma}

\begin{proof}
 Let $(\xi_1,\xi_2)\in\R^2$ be such that $|\xi_1|\sim N_1$ and $|\xi_2|\sim N_2$. First we estimate $\partial^{\beta}\Omega_2(\xi_1,\xi_2)$ for $\beta=(\beta_1,\beta_2)\in \mathbb{N}^2$.
 From Lemma \ref{res2} and the mean value theorem we easily get that
 \begin{align}
  |\partial^\beta \Omega_2(\xi_1,\xi_2)| \lesssim N_1N_2^{\alpha-\beta_2}\ &\textrm{ if } \beta_1=0, \beta_2\ge 0, \label{omega.inv.1}\\
  |\partial^\beta \Omega_2(\xi_1,\xi_2)| \lesssim N_2^{\alpha+1-|\beta|}\  &\textrm{ if } \beta_1\ge 1, \beta_2\ge 1 \textrm{ or } \beta=(1,0) ,\label{omega.inv.2} \\
  |\partial^\beta \Omega_2(\xi_1,\xi_2)| \lesssim N_1^{\alpha+1-|\beta|}\  &\textrm{ if } \beta_1\ge 2, \beta_2=0. \label{omega.inv.3}
 \end{align}
 Now classical derivative rules lead to
 $$
 \left|\partial^\beta\left( \frac 1{\Omega_2(\xi_1,\xi_2)}\right) \right| \lesssim \sum_{\gamma\in C_\beta} \frac 1{|\Omega_2(\xi_1,\xi_2)|^{|\beta|+1}} \prod_{\substack{0\le i\le \beta_1 \\ 0\le j\le \beta_2}} |\partial^{(i,j)} \Omega_2(\xi_1,\xi_2)|^{\gamma_{i,j}} ,
 $$
 where
 $$
 C_\beta = \left\{\gamma=(\gamma_{i,j})_{\substack{0\le i\le \beta_1 \\ 0\le j\le \beta_2}} : \sum_{\substack{0\le i\le \beta_1\\ 0\le j\le \beta_2}} \gamma_{i,j} = |\beta|, \sum_{\substack{0\le i\le \beta_1\\ 0\le j\le \beta_2}} i\gamma_{i,j} = \beta_1, \sum_{\substack{0\le i\le \beta_1\\ 0\le j\le \beta_2}} j\gamma_{i,j} = \beta_2 \right\}.
 $$
Therefore, we deduce from \eqref{resonance} as well as \eqref{omega.inv.1}-\eqref{omega.inv.2}-\eqref{omega.inv.3} that
\begin{align*}
 &|\xi_1|^{\beta_1} |\xi_2|^{\beta_2} |\partial^\beta \chi(\xi_1, \xi_2)| \\
 & \quad \lesssim \max_{\gamma\in C_\beta} \frac{N_1^{1+\beta_1} N_2^{\alpha+\beta_2}}{(N_1N_2^\alpha)^{|\beta|+1}} N_2^{\alpha\gamma_{1,0}} \prod_{j=0}^{\beta_2} (N_1N_2^{\alpha-j})^{\gamma_{0,j}} \cdot \prod_{i=2}^{\beta_1} N_1^{(\alpha+1-i)\gamma_{i,0}} \cdot \prod_{\substack{1\le i\le \beta_1\\ 1\le j\le \beta_2}} N_2^{(\alpha+1-i-j)\gamma_{i,j}}\\
 &\quad \lesssim \max_{\gamma\in C_\beta} N_1^{A_\gamma} N_2^{B_\gamma},
\end{align*}
with
$$
A_\gamma = \sum_{j=0}^{\beta_2} \gamma_{0,j} + \sum_{i=2}^{\beta_1} (\alpha+1-i)\gamma_{i,0} - \beta_2
$$
and
$$
B_\gamma = \alpha\gamma_{1,0} + \sum_{j=0}^{\beta_2}(\alpha-j)\gamma_{0,j} + \sum_{\substack{1\le i\le \beta_1\\ 1\le j\le \beta_2}} (\alpha+1-i-j)\gamma_{i,j} + \beta_2-\alpha |\beta|.
$$
Noticing that for $\gamma\in C_\beta$ we have
$$
 \beta_2 = |\beta|-\beta_1
= \sum_{\substack{0\le i\le \beta_1\\ 0\le j\le \beta_2}} \gamma_{i,j} - \sum_{\substack{0\le i\le \beta_1\\ 0\le j\le \beta_2}} i\gamma_{i,j}
 = \sum_{j=0}^{\beta_2} \gamma_{0,j} - \sum_{\substack{1\le i\le \beta_1\\ 0\le j\le \beta_2}} (i-1)\gamma_{i,j},
$$
we infer
\begin{align*}
 A_\gamma &= \alpha \sum_{i=2}^{\beta_1} \gamma_{i,0} - \sum_{i=2}^{\beta_1} (i-1)\gamma_{i,0} + \sum_{\substack{1\le i\le \beta_1\\ 0\le j\le \beta_2}} (i-1)\gamma_{i,j}\\
 &= \alpha \sum_{i=2}^{\beta_1} \gamma_{i,0} + \sum_{\substack{1\le i\le \beta_1\\ 1\le j\le \beta_2}} (i-1)\gamma_{i,j}.
\end{align*}
Similarly, we get
\begin{align*}
 B_\gamma &= \alpha\left( -|\beta| + \gamma_{0,0}+\gamma_{1,0} + \sum_{\substack{0\le i\le \beta_1\\ 1\le j\le \beta_2}} \gamma_{i,j}\right) + \left(\beta_2 - \sum_{\substack{0\le i\le \beta_1\\ 1\le j\le \beta_2}} j\gamma_{i,j} - \sum_{\substack{1\le i\le \beta_1\\ 1\le j\le \beta_2}} (i-1)\gamma_{i,j} \right)\\
 &= -\alpha \sum_{i=2}^{\beta_1} \gamma_{i,0} - \sum_{\substack{1\le i\le \beta_1\\ 1\le j\le \beta_2}} (i-1)\gamma_{i,j}.
\end{align*}
We conclude that $A_\gamma\ge 0$ and $A_\gamma=-B_\gamma$, which provides
$$
|\xi_1|^{\beta_1} |\xi_2|^{\beta_2} |\partial^\beta \chi(\xi_1, \xi_2)| \lesssim \max_{\gamma\in C_\beta} N_2^{A_\gamma+B_\gamma} \lesssim 1.
$$
\end{proof}

\subsection{Basic estimates on the sum space $F^{0,\frac12}=X^{-1,1}+X^{0,(\frac12)_+}$}
By definition of sum space in \eqref{F}, we always have by taking the trivial decompositions $(u_1,u_2)=(u,0)$ or $(u_1,u_2)=(0,u)$ that
\begin{equation} \label{Fprop}
\|u\|_{F^{0,\frac12}} \le \min\{ \|u\|_{X^{-1,1}} \, , \, \|u\|_{X^{0,(\frac12)_+}} \} \, .
\end{equation}
The next lemma tells us when the reverse holds true.

\begin{lemma}\label{eqFsb}
Let $u\in F^{0,\frac12}$ and $L, N$ be two dyadic numbers.

If $1\le L\lesssim N^2$, then
\begin{equation}\label{eqFsb.0}
\|Q_{\gtrsim L} u_N\|_{L^2_{x,t}} \lesssim NL^{-1} \|Q_{\gtrsim L} u_N\|_{F^{0,\frac12}} \, .
\end{equation}

If $L\gtrsim \cro{N}^2$, then
\begin{equation}
\|Q_{\gtrsim L} u_N\|_{L^2_{x,t}} \lesssim L^{-\frac12}\|Q_{\gtrsim L}u_N\|_{F^{0,\frac12}} \, . \end{equation}
\end{lemma}

\begin{proof}
  It directly follows from the estimate
  \begin{equation}\label{eqFsb.2}
    \|Q_{\gtrsim L} u_N\|_{L^2_{x,t}} \lesssim (L^{-[(\frac{1}{2})_+]}\vee L^{-1}N) \|Q_{\gtrsim L} u_N\|_{F^{0,\frac12}}.
  \end{equation}
\end{proof}

\section{$L^2$-multilinear estimates}

\subsection{$L^2$-bilinear estimates}
We follow the strategy in \cite{MV} to show $L^2$-bilinear estimates related to the dispersive symbol.

Let us define the resonance function of order 2 associated with \eqref{fKdV} by
\begin{equation} \label{Omega2}
\Omega_2(\xi_1, \xi_2) = \omega_{\alpha+1}(\xi_1+\xi_2) - \omega_{\alpha+1}(\xi_1)-\omega_{\alpha+1}(\xi_2)
\end{equation}
where $\omega_{\alpha+1}$ is the dispersive symbol defined in Hypothesis \ref{hyp1}. For $\xi_1, \xi_2, \xi_3\in\R$, it will be convenient to define the quantities $|\xi_{max}|\ge |\xi_{med}|\ge |\xi_{min}|$ to be the maximum, median and minimum of $|\xi_1|, |\xi_2|$ and $|\xi_3|$ respectively.

For the sake of completeness, we recall a few results proved in \cite{MV}.

\begin{lemma}[\cite{MV}, Lemma 2.1]\label{res2}
 Let $\alpha>0$. Let $\xi_1,\xi_2\in\R$, and $\xi_3=-(\xi_1+\xi_2)$. Then
  \begin{equation}\label{resonance}
    |\Omega_2(\xi_1,\xi_2)| \sim |\xi_{\min}| |\xi_{\max}|^\alpha.
  \end{equation}
\end{lemma}

\begin{lemma}[\cite{MV}, Lemma 2.3] \label{QLbound}
Let $L\ge 1$, $1\le p\le\infty$ and $s\in\R$. The operator $Q_{\le L}$ is bounded in $L^p_tH^s_x$ uniformly in $L\ge 1$.
\end{lemma}

For any $T>0$, we consider $1_T$ the characteristic function of the interval $]0,T[$ and use the decomposition
\begin{equation}\label{ind-dec}
1_T = 1_{T,R}^{low}+1_{T,R}^{high},\quad \widehat{1_{T,R}^{low}}(\tau)=\eta(\tau/R)\widehat{1_T}(\tau)
\end{equation}
for some $R>0$.

\begin{lemma}[\cite{MV}, Lemma 2.4]\label{ihigh-lem} For any $ R>0 $ and $ T>0 $, it holds
\begin{equation}\label{high}
\|1_{T,R}^{high}\|_{L^1}\lesssim T\wedge R^{-1}.
\end{equation}
and
\begin{equation}\label{low}
\|1_{T,R}^{low}\|_{L^\infty}\lesssim  1.
\end{equation}
\end{lemma}

\begin{lemma}[\cite{MV}, Lemma 2.5]\label{ilow-lem}
Let $u\in L^2(\R^2)$. Then for any $T>0$, $R>0$  and $ L \gg R $, it holds
$$
\|Q_L (1_{T,R}^{low}u)\|_{L^2}\lesssim \|Q_{\sim L} u\|_{L^2}
$$
\end{lemma}

We are now in a position to prove the main result of this section.

\begin{proposition}\label{L2bilin}
  Let $0<\alpha \le 1$. Assume  $0<t\le 1$ and $u_i\in Z^0$, $i=1,2,3$ are functions with spatial Fourier support in $\{|\xi|\sim N_i\}$ with $N_i$ dyadic. Let $\chi\in C^\infty(\R^2)$ satisfy the Marcinkiewicz condition (\ref{pseudoprod.1}).

If $N_{min}\lesssim 1$, then
  \begin{equation}\label{L2bilin.00}
   |G^2_{t,\chi}(u_1,u_2,u_3)| \lesssim N_{min}^{\frac12} \|u_1\|_{L^\infty_tL^2_x}\|u_2\|_{L^2_{tx}} \|u_3\|_{L^2_{tx}} \, .
  \end{equation}

 If $N_{min}\gg 1$, then
  \begin{equation}\label{L2bilin.0}
    |G^2_{t,\chi}(u_1,u_2,u_3)| \lesssim N_{min}^{-\frac 12-\frac{\alpha}4} N_{max}^{(1-\alpha)_+} \prod_{i=1}^3 \|u_i\|_{Z^0} \, ,
  \end{equation}
  where $G^2_{t,\chi}$ is defined in \eqref{Gtdef}.
\end{proposition}

\begin{proof}
From \eqref{Gtsym.2} we may always assume $N_1\le N_2\le N_3$.
Estimate \eqref{L2bilin.00} is easily obtained thanks to Plancherel identity and Bernstein inequality. Thus it remains to deal with the case $N_1\gg 1$.
By localization considerations, $G_{t,\chi}^2$ vanishes unless $N_2\sim N_3$. Setting $R = N_1^{1+\frac\alpha 4}N_3^{\alpha-1}$, we split $G^2_{t,\chi}$ as
\begin{align}
G^2_{t,\chi}(u_1,u_2,u_3) &= G^2_\infty(1_{t,R}^{high}u_1,u_2, u_3) + G^2_\infty(1_{t,R}^{low}u_1,u_2, u_3) \nonumber \\
& :=G_t^{2,high} + G_t^{2,low}, \label{L2bilin.1}
\end{align}
 where $G^2_\infty(u,v,w) = \int_{\R^2}\Pi_\chi^2(u,v)w$ and $1_{t,R}^{high}$, $1_{t,R}^{low}$ are defined in \eqref{ind-dec}.

The contribution of $G_t^{2,high}$ is estimated thanks to Lemma \ref{ihigh-lem} as well as H\"{o}lder inequality by
\begin{align}
|G_t^{2,high}| &\lesssim N_1^{\frac12} \|1_{t,R}^{high}\|_{L^1} \|u_1\|_{L^\infty_tL^2_x} \|u_2\|_{L^\infty_tL^2_x}\|u_3\|_{L^\infty_tL^2_x} \label{L2bilin.3} \\
&\lesssim N_1^{-\frac 12-\frac\alpha 4}N_3^{1-\alpha} \prod_{i=1}^3 \|u_i\|_{Z^0}. \nonumber
\end{align}
To evaluate the contribution of $G_t^{2,low}$, we use Lemma \ref{res2} and we get
\begin{align}
G_t^{2,low}  =  & G^2_\infty(Q_{{ \gtrsim } N_1 N_3^\alpha}( 1_{t,R}^{low}u_1),u_2,  u_3)\nonumber \\
& +G^2_\infty(Q_{{ \ll }N_1 N_3^\alpha} ( 1_{t,R}^{low}u_1), Q_{{ \gtrsim }N_1 N_3^\alpha}  u_2, u_3 )\nonumber \\
& + G^2_\infty(Q_{{ \ll }N_1 N_3^\alpha}  ( 1_{t,R}^{low}u_1), Q_{{ \ll }N_1 N_3^\alpha}  u_2, Q_{\sim N_1 N_3^\alpha}u_3)\nonumber \\
&=: G_{t,1}^{2,low}+G_{t,2}^{2,low}+G_{t,3}^{2,low}\; . \label{L2bilin.4}
\end{align}
It is worth noticing that since $ N_1\gg 1 $,  we have $R \ll N_1 N_3^\alpha $.
Therefore the  contribution of $  G_{t,1}^{2,low} $ is easily estimated thanks to Lemma \ref{ilow-lem}, Theorem \ref{pseudoprod} and estimate \eqref{eqFsb.2} by
\begin{align}
 |G_{t,1}^{2,low}| &\lesssim \|\Pi_\chi^2(Q_{{ \gtrsim } N_1 N_3^\alpha}( 1_{t,R}^{low}u_1),u_2)\|_{L^2_tL^1_x} \|u_3\|_{L^2_tL^\infty_x} \nonumber\\
 &\lesssim \|Q_{\gtrsim N_1N_3^\alpha}(1_{t,R}^{low}u_1)\|_{L^2_{tx}} \|u_2\|_{L^\infty_tL^2_x} \|u_3\|_{L^2_tL^\infty_x} \nonumber\\
 &\lesssim \sum_{L_1\gtrsim N_1N_3^\alpha} (L_1^{-(\frac12)_+} \vee L_1^{-1}N_1)  N_3^{\frac 12-\frac\alpha 4} \|u_1\|_{F^{0,\frac12}} \|u_2\|_{Z^0}\|u_3\|_{Z^0} \nonumber\\
 &\lesssim N_1^{-\frac 12-\frac \alpha 4}N_3^{1-\alpha} \prod_{i=1}^3 \|u_i\|_{Z^0} \label{L2bilin.5},
\end{align}
where in the last step we used that $ 0\le \alpha\le 1 $.
Using again Theorem \ref{pseudoprod}, H\"{o}lder inequality and Lemma \ref{eqFsb} we estimate the contribution of $G_{t,2}^{2,low}$ by
\begin{align}
|G_{t,2}^{2,low}| &\lesssim \|Q_{\ll N_1N_3^\alpha}(1_{t,R}^{low}u_1)\|_{L^2_tL^{\infty-}_x} \|Q_{\gtrsim N_1N_3^\alpha}u_2\|_{L^2_tL^{2+}_x} \|u_3\|_{L^\infty_tL^2_x} \nonumber\\
 &\lesssim (N_1N_3^\alpha)^{-1}N_3 N_3^{0_+} \|Q_{\ll N_1N_3^\alpha}(1_{t,R}^{low}u_1)\|_{L^2_tL^{\infty-}_x} \|u_2\|_{F^{0,\frac12}} \|u_3\|_{Z^0} \nonumber\\
 &\lesssim N_1^{-\frac 12-\frac\alpha 4} N_3^{(1-\alpha)_+} \left(N_1^{-\frac 12+\frac\alpha 4} \|Q_{\ll N_1N_3^\alpha}(1_{t,R}^{low}u_1)\|_{L^2_tL^{\infty-}_x}\right) \|u_2\|_{Z^0} \|u_3\|_{Z^0}. \label{L2bilin.6}
\end{align}
On the other hand, observe that an interpolation argument provides
 \begin{equation}\label{L2bilin.06}
  N^{(\frac \alpha 4-\frac 12)} \|u_N\|_{L^2_tL^{\infty-}_x} \lesssim  \|u_N\|_{Z^0} \ \textrm{ if } \ N\gtrsim 1.
  \end{equation}
Since $Q_{\ll L}=I-Q_{\gtrsim L}$, we deduce that
\begin{align}
 N_1^{-\frac 12+\frac\alpha 4} \|Q_{\ll N_1N_3^\alpha}(1_{t,R}^{low}u_1)\|_{L^2_tL^{\infty-}_x} &\lesssim N_1^{-\frac 12+\frac\alpha 4}\|u_1\|_{L^2_tL^{\infty-}_x} + N_1^{\frac\alpha 4} \|Q_{\gtrsim N_1N_3^\alpha} u_1\|_{L^2_{tx}} \nonumber\\
 &\lesssim  \|u_1\|_{Z^0} + \sum_{L\gtrsim N_1N_3^\alpha} (L^{-\frac12}\vee L^{-1}N_1) N_1^{\frac\alpha 4}\|u_1\|_{F^{0,\frac12}} \nonumber\\
 &\lesssim  \|u_1\|_{Z^0}. \label{L2bilin.7}
 \end{align}
Combining \eqref{L2bilin.6}-\eqref{L2bilin.7} we infer
$$
|G_{t,2}^{2,low}| \lesssim N_1^{-\frac 12-\frac\alpha 4}N_3^{(1-\alpha)_+} \prod_{i=1}^3 \|u_i\|_{Z^0}.
$$
Finally, using Lemma \ref{QLbound}, the contribution of $G_{t,3}^{2,low}$ is estimated in the same way.
\end{proof}

\subsection{$L^2$-trilinear estimates}
We first state an elementary estimate.
\begin{proposition}\label{eL2trilin}
Let $0<\alpha \le 1$. Assume $0<t\le 1$ and $u_i\in Z^0$  $i=1,2,3,4$ are functions with spatial Fourier support in $\{|\xi|\sim N_i\}$ with $N_i$ dyadic. Let $\chi\in C^\infty(\R^3)$ satisfy the Marcinkiewicz condition (\ref{pseudoprod.1}).

Then it holds that
\begin{equation}\label{eL2trilin.2}
 |G_{t,\chi}^3(u_1,u_2,u_3,u_4)| \lesssim N_1^{(\frac12)_-}\cro{N_1}^{-\frac \alpha 4} N_2^{(\frac12)_-}\cro{N_2}^{-\frac \alpha 4} N_{max}^{0_+} \prod_{i=1}^4 \|u_i\|_{Z^0}.
\end{equation}
\end{proposition}

\begin{proof}
  We get from \eqref{pseudoprod.2} together with H\"{o}lder and Bernstein inequalities that
  \begin{align*}
    |G_{t, \chi}^3(u_1,u_2,u_3,u_4)| &\lesssim \|\Pi_\chi^3(u_1,u_2,u_3)\|_{L^1_t L^{2}_x} \|u_4\|_{L^\infty_t L^2_x}\\
    &\lesssim  N_{max}^{0_+}\|u_1\|_{L^2_t L^{\infty-}_x} \|u_2\|_{L^2_t L^{\infty-}_x} \|u_3\|_{L^\infty_t L^2_x} \|u_4\|_{L^\infty_t L^2_x}.
  \end{align*}
  We conclude the proof of estimate \eqref{eL2trilin.2} combining
  \begin{equation}\label{eL2trilin.4}
  \|u_N\|_{L^2_t L^{\infty-}_x} \lesssim N^{(\frac12)_-} \|u_N\|_{L^\infty_t L^2_x} \ \textrm{ if } \ N\lesssim 1,
  \end{equation}
  with \eqref{L2bilin.06}.
  \end{proof}

Now we define the resonance function of order 3 by
$$
  \Omega_3(\xi_1,\xi_2,\xi_3) = \omega_{\alpha+1}(\xi_1+\xi_2+\xi_3) - \omega_{\alpha+1}(\xi_1) - \omega_{\alpha+1}(\xi_2) - \omega_{\alpha+1}(\xi_3).
$$
For $\xi_1, \xi_2, \xi_3, \xi_4\in\R$, it will be convenient to define the quantities $|\xi_{max}|\ge |\xi_{sub}|\ge |\xi_{thd}|\ge |\xi_{min}|$ to be the maximum, sub-maximum, third-maximum and minimum of $|\xi_1|, |\xi_2|, |\xi_3|$ and $|\xi_4|$ respectively.

\begin{lemma}
  Let $\alpha>0$. Let $\xi_1,\xi_2,\xi_3\in\R$ and $\xi_4=-(\xi_1+\xi_2+\xi_3)$. If we assume that $|\xi_{min}|\ll |\xi_{thd}|$ then it holds
  \begin{equation}\label{res3}
    |\Omega_3(\xi_1,\xi_2,\xi_3)| \sim |\xi_{thd}| |\xi_{max}|^\alpha.
  \end{equation}
\end{lemma}

\begin{proof}
  Without loss of generality, we may assume $|\xi_1|\ll |\xi_2| \le |\xi_3| \sim |\xi_4|$. Then, estimate \eqref{res3} is a consequence of the identity
  $$
  \Omega_3(\xi_1,\xi_2,\xi_3) = \Omega_2(\xi_2+\xi_3, \xi_1) + \Omega_2(\xi_2,\xi_3)
  $$
  combined with Lemma \ref{res2}.
\end{proof}

\begin{proposition}\label{L2trilin}
Let $0<\alpha \le 1$. Assume $0<t\le 1$ and $u_i\in Z^0$, $i=1,2,3,4$ are functions with spatial Fourier support in $\{|\xi|\sim N_i\}$ with $N_i$ dyadic satisfying $N_{min}\ll N_{thd}$ and $N_{max}\gg 1$.  Let $\chi\in C^\infty(\R^3)$ satisfy the Marcinkiewicz condition (\ref{pseudoprod.1}). Then,
\begin{equation}\label{L2trilin.0}
    |G^3_{t,\chi}(u_1,u_2,u_3,u_4)| \lesssim N_{min}^{(\frac12)_-} \cro{N_{thd}}^{-\frac12-\frac \alpha 4} N_{max}^{(1-\alpha)_+} \prod_{i=1}^4 \|u_i\|_{Z^0}.
\end{equation}
\end{proposition}

\begin{proof}
  From \eqref{Gtsym.2} it is sufficient to consider the case $N_1\ll N_2\le N_3 \sim N_4$.
  Moreover, we may assume that $N_2N_4^\alpha\gg 1$ and $N_2\gg 1$ since otherwise the claim follows from estimate \eqref{eL2trilin.2}.
  We proceed now as in the proof of Proposition \ref{L2bilin}. First we decompose $G_t^3$ as $G_t^{3,high}+G_t^{3,low}$ with
  $$
  G_t^{3,high}(u_1,u_2,u_3,u_4) = G_\infty^3(1_{t,R}^{high} u_1, u_2,u_3,u_4) = \int_{\R^2} \Pi_\chi^3(1_{t,R}^{high} u_1, u_2,u_3)u_4 \, dxdt
  $$
  and $R = N_2^{1+\frac \alpha 4} N_4^{\alpha-1}\ll N_2N_4^\alpha$. The high-part is easily estimated thanks to Lemma \ref{ihigh-lem} by
  \begin{align}
    |G_t^{3,high}| &\lesssim R^{-1} N_1^{\frac12}N_2^{\frac12} \prod_{i=1}^4 \|u_i\|_{L^\infty_tL^2_x},
  \end{align}
  which is acceptable. To deal with the low-part, we decompose with respect of the modulation variables. Thus
  $$
  G_t^{3, low} = \sum_{L_1, L_2, L_3, L_4} G_\infty^3(Q_{L_1}(1_{t, R}^{low}u_1), Q_{L_2}u_2, Q_{L_3}u_3, Q_{L_4}u_4).
  $$

  According to \eqref{res3} the above sum is nontrivial only for $L_{max}\gtrsim N_2N_4^\alpha$. In the case where $L_{max}=L_1$, we deduce from \eqref{pseudoprod.2}-\eqref{eqFsb.2}-\eqref{L2bilin.06} and Lemma \ref{ilow-lem} that
  \begin{align*}
    |G_{t,1}^{3,low}| &\lesssim N_1^{(\frac12)_-} N_4^{0_+} \|Q_{\gtrsim N_2N_4^\alpha}(1_{t,R}^{low}u_1)\|_{L^2_{t,x}} \|u_2\|_{L^2_tL^{\infty-}_x} \|u_3\|_{L^\infty_tL^2_x} \|u_4\|_{L^\infty_tL^2_x}\\
    &\lesssim \sum_{L_1\gtrsim N_2N_4^\alpha} (L_1^{-(\frac12)_+} \vee L_1^{-1}N_1) N_1^{(\frac12)_-} N_2^{\frac 12-\frac \alpha 4} N_4^{0_+} \|u_1\|_{F^{0,\frac12}} \|u_2\|_{Z^0} \|u_3\|_{Z^0} \|u_4\|_{Z^0}\\
    &\lesssim N_1^{(\frac12)_-} N_2^{-\frac 12-\frac \alpha 4} N_4^{(1-\alpha)_+} \prod_{i=1}^4 \|u_i\|_{Z^0}.
  \end{align*}
  In the same way, we get that the sum over $L_{max}=L_2$ is controlled by
  \begin{align*}
    |G_{t,2}^{3,low}| &\lesssim N_1^{(\frac12)_-} N_4^{0_+} \|Q_{\ll N_2N_4^\alpha}(1_{t,R}^{low}u_1)\|_{L^\infty_tL^2_x} \|Q_{\gtrsim N_2N_4^\alpha}u_2\|_{L^2} \|u_3\|_{L^2_tL^{\infty-}_x} \|u_4\|_{L^\infty_tL^2_x}\\
    &\lesssim \sum_{L_2\gtrsim N_2N_4^\alpha} (L_2^{-(\frac12)_+} \vee L_2^{-1}N_2) N_1^{(\frac12)_-} N_4^{(\frac 12-\frac \alpha 4)+} \|u_1\|_{L^\infty_tL^2_x} \|u_2\|_{F^{0,\frac12}} \|u_3\|_{Z^0} \|u_4\|_{Z^0}\\
    &\lesssim N_1^{(\frac12)_-} N_2^{-\frac 12-\frac \alpha 4} N_4^{(1-\alpha)_+} \prod_{i=1}^4 \|u_i\|_{Z^0}.
  \end{align*}
  Arguing similarly and using \eqref{L2bilin.7}, the sum over $ L_{max}=L_3 $ can be estimated by
  \begin{align*}
    |G_{t,3}^{3,low}| &\lesssim N_1^{(\frac12)_-} N_4^{0_+} \|Q_{\ll N_2N_4^\alpha}(1_{t,R}^{low}u_1)\|_{L^\infty_tL^2_x} \|Q_{\ll N_2N_4^\alpha}u_2\|_{L^2_tL^{\infty-}_x} \|Q_{\gtrsim N_2N_4^\alpha}u_3\|_{L^2} \|u_4\|_{L^\infty_tL^2_x}\\
    &\lesssim \sum_{L_3\gtrsim N_2N_4^\alpha} (L_3^{-(\frac12)_+} \vee L_3^{-1}N_4) N_1^{(\frac 12)-} N_2^{\frac 12-\frac \alpha 4} N_4^{0_+} \|u_1\|_{L^\infty_tL^2_x} \|u_2\|_{Z^0} \|u_3\|_{F^{0,\frac12}} \|u_4\|_{Z^0}\\
    &\lesssim N_1^{(\frac12)_-} N_2^{-\frac12-\frac \alpha 4} N_4^{(1-\alpha)_+} \prod_{i=1}^4 \|u_i\|_{Z^0}.
  \end{align*}
  Finally we easily check that the bound in the case $L_{max}=L_4$ is obtained similarly. Gathering all these estimates we get the desired result.

\end{proof}

\section{Estimates for a smooth solution}
The aim of this section is to get suitable a priori estimates of a solution of \eqref{dpB} in the space $Y^s\hookrightarrow Z^s$ for $s>s_\alpha$.
\subsection{Bilinear estimate}
\begin{proposition}
Assume that $0<T\le 1$ and $s\ge 0$. Let $u$ be a smooth solution to (\ref{dpB}) defined in the time interval $[0,T]$. Then
\begin{equation} \label{be.0}
\|u\|_{X^{s-1,1}_T} \lesssim \|u_0\|_{H^s} + \|u\|_{L^2_TL^\infty_x} \|u\|_{L^\infty_TH^s_x}.
\end{equation}
\end{proposition}
\begin{proof}
By using  the fractional Leibniz rule (c.f. Kenig, Ponce and Vega \cite{KPV}), we have for $s \ge 0$
\begin{equation} \label{be}
\begin{split}
\|u\|_{X^{s-1,1}_T} &\lesssim \|u_0\|_{H^{s-1}}+ \|\partial_x(u^2)\|_{X^{s-1,0}_T} \\
&  \lesssim \|u_0\|_{H^{s}}+\|J^s_x(u^2)\|_{L^2_{x,T}}  \\ & \lesssim \|u_0\|_{H^{s}}+\|u\|_{L^2_TL^{\infty}_x}\|u\|_{L^{\infty}_TH^s_x} \, .
\end{split}
\end{equation}
\end{proof}

\subsection{Refined Strichartz estimate}
Let us first recall the following Strichartz estimate:
\begin{equation} \label{strichartz} 
\|P_{\ge 1}D_x^{(\alpha-1)/4}U_\alpha(t)u_0\|_{L^4_tL^\infty_x} \lesssim \|u_0\|_{L^2},\quad u_0\in L^2(\R),
\end{equation}
where $U_\alpha(t)=e^{tL_{\alpha+1}}$ is the free evolution operator associated to \eqref{dpB}. This estimate is a direct consequence of Theorem 2.1 in \cite{KPV2} applied with $\phi = (1-\eta)\omega_{\alpha+1}$. From this we get following the proof of Proposition 2.3 in \cite{LiPiSa1} (see also \cite{KK}) the refined Strichartz estimate:
\begin{lemma}\label{se}
Let $0<\alpha \le 1$.  Assume that $0<T\le 1$ and $\delta\ge 0$. Let $u$ be a  solution to
  \begin{equation}\label{se.0}
  (\partial_t+L_{\alpha+1})u = F
  \end{equation}
  defined on the time interval $[0,T]$. Then, there exist  $ 0<\kappa_1, \; \kappa_2<\frac 1 2 $ such that   \begin{equation}\label{se.1}
    \|P_Nu\|_{L^2_TL^{\infty}_x} \lesssim T^{\kappa_1} \|D_x^{-(\alpha-1)/4+\delta/4}P_Nu\|_{L^{\infty}_TL^2_x}
+T^{\kappa_2} \|D_x^{-(\alpha-1)/4-3\delta/4}P_NF\|_{L^2_{T,x}}
  \end{equation}
  and 
   \begin{equation}\label{se.11}
    \|P_Nu\|_{L^2_TL^{\infty}_x} \lesssim T^{\kappa_1} \|D_x^{-(\alpha-1)/4+\delta/4}P_Nu\|_{L^{\infty}_TL^2_x}
+T^{\kappa_2} \|D_x^{-(\alpha-1)/2-\delta/2}P_NF\|_{L^2_T L^1_x},
  \end{equation}
for any dyadic number $N\ge 1$.
\end{lemma}
\begin{proof}
\eqref{se.1} is proven in  [\cite{LiPiSa1}, Proposition 2.3] (see also \cite{KK}). To prove \eqref{se.11} we   modified slightly the procedure (see \cite{MoPiVe1} for a similar modification). Let $ N\ge 1 $ and let $ I=[a,b]\subset \R $ be an interval of length 
 $ |I|\lesssim T^\kappa N^{-\delta} $ for some  fixed $ \delta>0 $ and $ 0<\kappa<1$. From \eqref{strichartz} and H\"older's inequalities, we easily get 
 \begin{equation}\label{ret1}
 \|U_\alpha(\cdot) P_N u_0 \|_{L^p_{I} L^\infty_x} \lesssim N^{\frac{1-\alpha}{4}} (T^\kappa N^{-\delta})^{(\frac{1}{p}-\frac{1}{4})} \|u_0\|_{L^2} \; , 
 \end{equation}
  for any $ 2\le p\le 4 $ and $ u_0\in L^2(\R) $. By the $T T^* $ method and P. Tomas argument, this leads to 
 $$
 \Bigl\| \int_{I} U_\alpha(\cdot-t') P_N f(t') dt' \|_{L^p_{I} L^\infty_x} \lesssim N^{\frac{1-\alpha}{2}} (T^\kappa N^{-\delta})^{(\frac{1}{p}+\frac{1}{p'}-\frac{1}{2})} \|f\|_{L^{\overline{p'}}_{I} L^1_x}\; , 
 $$
  with $ \overline{p'}=\frac{p'}{p'-1} $, for any $ 2\le p,p' \le 4 $ and any $ f\in L^{\overline{p'}}_{I} L^1_x $. We need this estimate but on the retarded Duhamel operator 
  $(t,x) \mapsto  \int_a^t U(t-t') P_N f(t',x) dt' $.  Taking $ p=2 $ and $p'>2 $, this can be done by applying Christ-Kiselev Lemma (see \cite{CK} and also \cite{SS}). We then get 
    $$
 \Bigl\| \int_a^t  U_\alpha(t-t') P_N f(t') dt' \Bigr\|_{L^2_{I} L^\infty_x} \lesssim N^{\frac{1-\alpha}{2}} (T^\kappa N^{-\delta})^{\frac{1}{p'}} \|f\|_{L^{\overline{p'}}_{I} L^1_x}\; , 
 $$
 and H\"older inequalities then yields 
 \begin{equation}\label{ret2}
 \Bigl\| \int_a^t  U_\alpha (t-t') P_N f(t') dt' \Bigr\|_{L^2_{I} L^\infty_x} \lesssim N^{\frac{1-\alpha}{2}} T^\frac{\kappa}{2} N^{-\frac{\delta}{2}} \|f\|_{L^2_{I} L^1_x}\; .
 \end{equation}
 Now,  chopping out the interval $ [0,T] $ in small intervals of length $ T^\kappa N^{-\delta} $, we have $ [0,T]=\cup_{j\in J} I_j $ where $ I_j=[a_j,b_j] $, $ |I_j|\sim 
  T^\kappa N^{-\delta} $ and $ \#J\sim T^{1-\kappa} N^\delta $. Since $ u_N $ satisfies $ \partial_t u_N -L_{\alpha+1} u_N=F_N $ on each interval $ I_j $ we have 
  \begin{align*}
  \|u_N\|_{L^2_T L^\infty_x}& = \Bigl( \sum_{j\in J} \|u_N\|_{L^2_{I_j} L^\infty_x}^2\Bigr)^{\frac12} \\ 
  & \lesssim \Bigl( \sum_{j\in J} \| U_{\alpha}(t-a_j) u_N(a_j) \|_{L^2_{I_j} L^\infty_x}^2 + 
 \sum_{j\in J}  \Bigl\|\int_{a_j}^t U(t-t') F(t') dt' \Bigr\|_{L^2_{I_j} L^\infty_x}^2 \Bigr)^{\frac12} 
    \end{align*} 
  and \eqref{ret1}-\eqref{ret2} yield 
   \begin{align*}
  \|u_N\|_{L^2_T L^\infty_x}& \lesssim    N^{\frac{1-\alpha}{4}} (T^\kappa N^{-\delta})^{\frac14}  \Bigl( \sum_{j\in J} \| u_N \|_{L^\infty_{T} L^2_x}^2\Bigr)^{\frac12}\\ & \quad + 
 N^{\frac{1-\alpha}{2}} T^\frac{\kappa}{2} N^{-\frac{\delta}{2}} \Bigl( \sum_{j\in J} \int_{I_j} \|F_N(t,\cdot)\|_{L^1_x}^2 \, dt\Bigr)^{\frac12} \\
 & \lesssim T^{\frac{1}{2}-\frac{\kappa}{4}} N^{\frac{1-\alpha+\delta}{4}} \|u_N \|_{L^\infty_T L^2_x} +T^\frac{\kappa}{2} N^\frac{1-\alpha-\delta}{2} \|F_N\|_{L^2_T L^1_x} \, ,
    \end{align*} 
  which leads to \eqref{se.11} by Bernstein inequalities.
  \end{proof}
\begin{proposition}\label{propse}
Let $0<\alpha \le 1$. Assume that $0<T\le 1$ and $s>s_\alpha$. Let $u$ be a smooth solution to \eqref{dpB} defined on the time interval $[0,T]$. There exists  $ 0< \kappa<\frac 1 2 $ such that if $0< T\ll \|u\|_{L^\infty_{T} H^s_x}^{-\frac{1}{\kappa}}$, then
\begin{equation}\label{se.2}
\|J_x^{(s-s_\alpha)+(1-\alpha)_-} u\|_{L^2_T L^\infty_x}  \le  2  T^{\kappa}\|u\|_{L^\infty_T H^s_x} \le 1 \; .
\end{equation}
\end{proposition}
\begin{proof}
From Bernstein's inequality, we easily estimate the low frequencies part:
$$
\|P_{\le 1}J_x^{(s-s_\alpha)+(1-\alpha)_-}u\|_{L^2_TL^\infty_x}\lesssim T^{\frac12}\|u\|_{L^\infty_TL^2_x}.
$$
Taking $\delta=1$ in \eqref{se.1}, summing over $N \ge 1$ and using the fractional Leibniz rule, we deduce
\begin{align*}
  \|P_{>1}J_x^{(s-s_\alpha)+(1-\alpha)_-} u\|_{L^2_T L^\infty_x} &\lesssim T^{\kappa_1}\|J_x^su\|_{L^\infty_TL^2_x} + T^{\kappa_2}\|J_x^s(u^2)\|_{L^2_{T,x}}\\
  &\lesssim  T^{\kappa_1} \|u\|_{L^\infty_TH^s_x} + T^{\kappa_2}\|u\|_{L^2_TL^\infty_x} \|u\|_{L^\infty_TH^s_x}.
\end{align*}
Noticing that for $ s>s_\alpha $ and $ 0<\alpha\le 1 $, it holds $(s-s_\alpha)+(1-\alpha)_-\ge 0 $, we obtain \eqref{se.2} by combining the two above estimates and taking $ \kappa=\kappa_1\vee \kappa_2$.
\end{proof}
\begin{corollary}\label{coro1}
Let $0<\alpha \le 1$. Assume that $0<T\le 1$ and $s>s_\alpha$. Let $u$ be a smooth solution to \eqref{dpB} defined on the time interval $[0,T]$. There exist  $ 0< \kappa<\frac 1 2 $ and $ C_0>1 $ such that if $0< T\ll \|u\|_{L^\infty_{T} H^s_x}^{-\frac{1}{\kappa}}$, then
\begin{equation}\label{co1}
\|u\|_{Y^s_T}   \le C_0   \|u\|_{L^\infty_T H^s_x} \; .
\end{equation}
\end{corollary}
\begin{proof}
We have to extend the function $ u$ from $ ]0,T[ $ to $ \R $. For this we introduce the extension operator $ \rho_T $ defined by
\begin{equation}\label{defrho}
\rho_T(u)(t):= U_\alpha(t)\eta(t) U_\alpha(-\mu_T(t)) u(\mu_T(t))\; ,
\end{equation}
where $ \eta $ is the smooth cut-off function defined in Section \ref{notation} and $\mu_T $ is the  continuous piecewise affine
 function defined  by
$$
 \mu_T(t)=\left\{\begin{array}{rcl}
 0  &\text{for } &  t<0 \\
 t  &\text {for }& t\in [0,T] \\
  T & \text {for } & t>T
 \end{array}
 \right. .
 $$
 According to classical results on extension operators (see for instance \cite{LM}),  for any $ 1/2<b\le 1$,  $f \mapsto \eta f(\mu_T(\cdot)) $ is  linear continuous    from $ H^b([0,T]) $ into $ H^b(\R) $  with a bound that does not depend on $ T>0 $.\\
First, the unitarity of the free group $ U_\alpha(\cdot) $ in $ H^s(\R) $ easily leads to
\begin{equation}\label{tg1}
\|\rho_T(u)\|_{L^\infty_t H^s_x} \lesssim \|u(\mu_T(\cdot))\|_{L^\infty_t H^s_x}\lesssim \|u\|_{L^\infty_T  H^s_x} + \|u(0)\|_{H^s}+ \|u(T)\|_{H^s} \; .
\end{equation}
Second, the definition of the $ X^{\theta,b}$-norm leads, for $1/2<b\le 1 $ and   $\theta\in \R $, to
\begin{equation}\label{tg2}
\|\rho_T(u)\|_{X^{\theta,b}}  = \|\eta\, U_\alpha(-\mu_T(\cdot)) u(\mu_T(\cdot))\|_{H^{{\theta},b}_{x,t} }
 \lesssim  \|U_\alpha(-\cdot) u\|_{H^b([0,T[; H^{\theta})}
 \lesssim  \|u\|_{X^{\theta,b}_T} \; .
\end{equation}
Finally,  for $ \theta \in \R $,
\begin{align*}
\|J^\theta_x \rho_T(u)\|_{L^2_t L^\infty_x} & \lesssim  \|\eta U_\alpha(-\cdot)J^\theta_x   u(0)\|_{L^2(]-\infty,0[; L^\infty_x)} + \|J^\theta_x u \|_{L^2_T L^\infty_x} 
 \\ 
 & + \|\eta U_\alpha(-\cdot)J^\theta_x  U_\alpha(T) u(T)\|_{L^2(]T,+\infty[; L^\infty_x)}
\end{align*}
whereas \eqref{strichartz} leads to 
\begin{align*}
  \|\eta U_\alpha(-\cdot)J^\theta_x   u(0)\|_{L^2(]-\infty,0[; L^\infty_x)} & \lesssim \| P_{\le 1} \eta U_\alpha(-\cdot)J^\theta_x   u(0)\|_{L^2_t H^1_x} 
 + \|P_{>1}   U_\alpha(-\cdot)J^\theta_x   u(0)\|_{L^4_t L^\infty_x}\\
 & \lesssim \|  u(0)\|_{L^2_x} + \|J^{\theta+\frac{1-\alpha}{4}}_x   u(0)\|_{L^2_x}\lesssim \|u(0)\|_{H^{\theta+\frac{1-\alpha}{4}}}
 \end{align*}
 and in the same way
 \begin{align*}
  \|\eta U_\alpha(-\cdot) U(T)J^\theta_x   u(T)\|_{L^2(]T,+\infty[; L^\infty_x)} 
 & \lesssim \|U(T)u(T)\|_{H^{\theta+\frac{1-\alpha}{4}}}=\|u(T)\|_{H^{\theta+\frac{1-\alpha}{4}}}\; .
 \end{align*}
Noticing that, for $ 0<\alpha\le 1 $, $ s-s_\alpha +1-\alpha=s-\frac12 +\frac\alpha 4 \le s-\frac14$, this ensures that 
\begin{eqnarray}
\|J_x^{(s-s_\alpha)+(1-\alpha)-\varepsilon} \rho_T(u)\|_{L^2_t L^\infty_x}&\lesssim &\|J^{(s-s_\alpha)+(1-\alpha)-\varepsilon} u \|_{L^2_T L^\infty_x} \nonumber\\
& & +\|u(0)\|_{H^{s-\varepsilon}}+\|u(T)\|_{H^{s-\varepsilon}}\; ,
\label{tg3}
\end{eqnarray}
for any $ \varepsilon>0 $.\\
Gathering \eqref{tg1}-\eqref{tg3}, we thus  infer that for any $ (T,s)\in \R_+^*\times \R $, $ \rho_T $ is a bounded linear operator from
 $ C([0,T];H^s(\R))  \cap X^{s-1,1}_T \cap  L^2_T W^{(s-s_\alpha)+(1-\alpha)_-,\infty}_x$ into $ Y^s $   with a bound that does not depend on $ (T,s)$.
 Therefore \eqref{be.0} and \eqref{se.2} lead to \eqref{co1}.
\end{proof}
\subsection{Energy estimate}\label{section-ee}
Applying the operator $ P_N $ with $ N>0  $ dyadic  to equation \eqref{dpB}, taking the $L^2$ scalar product with $ P_N u $ and integrating
 on $ ]0,t[ $  we obtain
\begin{equation}\cro{N}^{2s}
\|P_N u(\cdot,t)\|^2_{L^2_x} = \|P_N u_0\|_{H^s}^2 +  \langle N\rangle^{2s} \int_{\R_t}  P_N\partial_x(u^2) P_N u
\end{equation}
Let $N_0\ge 2^9$ and $N>N_0$. Define $\mathcal{J}_N$ by
\begin{equation}\label{JN}
\mathcal{J}_N = \cro{N}^{2s}\int_{\R_t} P_N\partial_x(u^2) P_N u.
\end{equation}
By localization considerations, we get
$$
P_N(u^2) = 2P_N(u_{\ll N}u) + P_N(u_{\gtrsim N}u_{\gtrsim N}).
$$
Moreover, from the fundamental theorem of calculus, we easily get
$$
P_N(u_{\ll N}u) = u_{\ll N} u_N +  N^{-1}\Pi_\chi^2(\partial_xu_{\ll N}, u),
$$
where we used the bilinear Fourier multiplier notation introduced in Definition \ref{defpseudo} with $$\chi(\xi_1,\xi_2)=-i\int_0^1\phi'(N^{-1}(\theta\xi_1+\xi_2)) d\theta \, .$$
Inserting this into (\ref{JN}) and integrating by parts we deduce $\mathcal{J}_N=\mathcal{J}_N^1+\mathcal{J}_N^2$ where
$$
  \mathcal{J}_N^1 = \cro{N}^{2s}\int_{\R_t} \left( \partial_x u_{\ll N} P_Nu +2N^{-1} \partial_x\Pi_\chi^2(\partial_x u_{\ll N}, u) \right) P_Nu
$$
and
\begin{equation} \label{defJN2}
\mathcal{J}_N^2 = -\cro{N}^{2s}\sum_{N_1\gtrsim N} \int_{\R_t} P_{N_1}uP_{\sim N_1}u P_N^2\partial_x u,
\end{equation}
Since $P_NP_{\sim N}=P_N$, we may rewrite $J_N^1$ more symmetrically as
\begin{align}
  \mathcal{J}_N^1 &= \cro{N}^{2s}\int_{\R_t} P_N\left( \partial_x u_{\ll N} P_NP_{\sim N} u +2N^{-1} \partial_x\Pi_\chi^2(\partial_x u_{\ll N}, P_{\sim N}^2u) \right) P_{\sim N}u \nonumber \\
  &= N^{2s} \int_{\R_t} \Pi^2_{\chi_1}(\partial_x u_{\ll N}, P_{\sim N} u) P_{\sim N} u \label{defJN1}
\end{align}
with
\begin{equation}\label{chi1}
\chi_1(\xi_1,\xi_2) = \Bigl( \frac{\cro{N}}{N}\Bigr)^{2s}\left(\phi_N(\xi_2) + 2i\frac{\xi_1+\xi_2}N \chi(\xi_1,\xi_2) \phi_{\sim N}(\xi_2)\right) \phi_N(\xi_1+\xi_2).
\end{equation}

Note that the function $\chi_1$ satisfies the condition \eqref{pseudoprod.1}. This decomposition of $\mathcal{J}_N$ motivates the definition of our modified energy. For  $ N_0>1$, $u\in H^s(\R) $, with $ s\in\R $, and $N>0$ dyadic we define
\begin{equation} \label{defEN}
  \mathcal{E}_N(u)=\mathcal{E}_N(u,N_0)= \left\{ \begin{array}{ll} \frac 12 \|P_Nu\|_{L^2_x}^2 & \text{for} \ N \le N_0 \, \\
\frac 12 \|P_Nu\|_{L^2_x}^2 + c \mathcal{E}_N^{1}(u)  & \text{for} \ N> N_0 \, , \end{array}\right.
\end{equation}
where
$$
\mathcal{E}_N^1(u) = \int_{\R^2}\frac{\chi_1(\xi_1,\xi_2)}{\Omega_2(\xi_1,\xi_2)} \xi_1 \widehat{u_{\ll N}}(\xi_1) \widehat{P_{\sim N}u}(\xi_2)\widehat{P_{\sim N}u}(-\xi_1-\xi_2) d\xi_1d\xi_2 \, ,
$$
$\Omega_2(\xi_1,\xi_2)$ is the quadratic resonance relation defined in \eqref{Omega2}, and $c$ is a real constant to be fixed later.

We define the  modified energy at the $ H^s$-regularity  by using a nonhomogeneous dyadic decomposition in spatial frequency\footnote{This means that when summing over $N$, we keep all the low frequencies together and by convention $P_1=P_{\le 1}$.}
\begin{equation}\label{def-EsT}
E^s(u) = E^s(u,N_0) =\sum_{N \ge 1} \cro{N}^{2s}  \big|\mathcal{E}_N(u,N_0)\big| \, .
\end{equation}

Next, we show that if $s >s_\alpha$ and $ N_0>2^0 $ is large enough then the modified energy $E^s(u)$ is equivalent to the $ H^s$-norm of $ u$.
\begin{lemma}[Coercivity of the modified energy]\label{lem-EsT}
Let $0<\alpha \le 1$ and let $u\in H^s(\R) $ with  $s > s_\alpha$. Then  for any $ N_0\gg (1+\|u\|_{H^{s_\alpha}_x})^{\frac 2 \alpha}  $, it holds
\begin{equation} \label{lem-Est.1}
\Bigl|E^s(u,N_0)-\frac{1}{2}\sum_{N \ge 1} \cro{N}^{2s} \|P_N u \|_{L^2_x}^2  \Bigr| \le \frac{1}{8}\sum_{N>N_0} \cro{N}^{2s} \|P_N u \|_{L^2_x}^2  \; .
\end{equation}
\end{lemma}

\begin{proof}
We infer from (\ref{def-EsT}) and the triangle inequality that
\begin{equation} \label{lem-Est.2}
\Bigl|E^s(u,N_0)-\frac12\sum_{N \ge 1} \cro{N}^{2s} \|P_N u \|_{L^2_x}^2  \Bigr| \lesssim  \sum_{N> N_0}N^{2s}\big|\mathcal{E}_N^{1}(u)\big| \, .
\end{equation}
Thanks to Young and Bernstein's inequalities we have
\begin{equation} \label{lem-Est.3}
\begin{split}
N^{2s}\big|\mathcal{E}_N^{1}(u)\big| &\lesssim \sum_{N_1\ll N} N^{2s} (N_1N^\alpha)^{-1} N_1^{\frac12} \|\partial_xu_{N_1}\|_{L^2_x} \|P_{\sim N}u\|_{L^2_x}^2 \\ &
\lesssim (N^{-\alpha}+N^{\frac{\alpha}4-1}) \|u\|_{H^{s_\alpha}_x} \|P_{\sim N}u\|_{H^s_x}^2 \, .
\end{split}
\end{equation}
Finally, we conclude the proof of \eqref{lem-Est.1} gathering \eqref{lem-Est.2}-\eqref{lem-Est.3} and the fact that
 $  \displaystyle \sum_{N \ge 1}\|P_{\sim N} u\|_{H^s_x}^2\sim \displaystyle \sum_{N \ge 1} \cro{N}^{2s} \|P_N u \|_{L^2_x}^2 \, . $
\end{proof}

We now state the main estimate of this subsection.
\begin{proposition}\label{prop-ee}
Let $0<\alpha \le 1$.  Let $s>s'>s_\alpha$, $0<T \le 1$ and $u\in Y^s_T$ be a solution of \eqref{dpB} on $[0,T]$. Then for any $ N_0\gg 1 $ we have
\begin{equation} \label{prop-ee.1}
\sup_{t\in ]0,T[}E^s(u(t),N_0) \lesssim E^s(u_0,N_0) + (TN_0^{\frac 3 2 }+ N_0^{(s_\alpha-s')_+}+N_0^{-\alpha_+} )(\|u\|_{Y^{s'}_T}+\|u\|_{Y^{s'}_T}^2) \|u\|_{Y^s_T}^2 \, ,
\end{equation}
where the implicit constant only depends on $ \alpha $.
\end{proposition}

\begin{proof}
Let $0<t \le T \le 1$.
First, assume that $N \le N_0=2^9$. By using the definition of $\mathcal{E}_N$ in \eqref{defEN}, we have
\begin{displaymath}
\frac{d}{dt} \mathcal{E}_N(u(t)) = \int_{\mathbb R}P_N\partial_x( u^2) P_Nu \, dx \, ,
\end{displaymath}
which yields after integrating between $0$ and $t$ and applying H\"older and Bernstein's inequalities  that
\begin{displaymath}
\begin{split}
|\mathcal{E}_N(u(t))| &\le |\mathcal{E}_N(u_0)|+\Big|\int_{\mathbb R_t}P_N\partial_x( u^2) P_Nu  \, \Big| \\ &
\lesssim  |\mathcal{E}_N(u_0)| + t\, N^{\frac32} \|P_N(u^2)\|_{L^\infty_TL^{1}_x} \|P_Nu\|_{L^\infty_TL^2_x} \, .
\end{split}
\end{displaymath}
Thus, we deduce after taking the supreme over $t \in [0,T[$ and summing over $N \le N_0$ that
\begin{equation} \label{prop-ee.2}
\sup_{t \in ]0,T[}\sum_{N \le N_0} \cro{N}^{2s}  \big|\mathcal{E}_N(u(t)) \big| \lesssim   \sum_{N \le N_0} \cro{N}^{2s}  \big|\mathcal{E}_N(u_0)|+ T\,N_0^{\frac32 }\|u\|_{L^\infty_TL^2_x} \|u\|_{L^\infty_T H^s_x}^2 \, ,
\end{equation}
where we used that, since $ s>0$, $ N^s \|P_N(u^2)\|_{L^\infty_T L^1_x}\lesssim \|u\|_{L^\infty_T L^2_x} \|u\|_{L^\infty_T H^s_x} $.

Now, for $ N\ge N_0$, we take the extension $ \tilde{u} =\rho_T(u)$ defined in \eqref{defrho}. To simplify the notation we drop the tilde in the sequel.
We first notice that
\begin{align}
\cro{N}^{2s}\mathcal{E}_N(u(t)) &= \cro{N}^{2s}\mathcal{E}_N(u_0) + \cro{N}^{2s}\int_{\R_t}P_N\partial_x(u^2)P_Nu + c \cro{N}^{2s}\int_0^t\frac{d}{dt}\mathcal{E}_N^{1}(t')dt' \nonumber \\
&=: \cro{N}^{2s}\mathcal{E}_N(u_0) + \mathcal{J}_N(t) + c \mathcal{K}_N(t)\, , \label{prop-ee.3}
\end{align}
where $\mathcal{J}_N(t)=\mathcal{J}_N^1(t)+\mathcal{J}_N^2(t)$ is defined in \eqref{JN}, \eqref{defJN2}, \eqref{defJN1}.
\medskip

\noindent \textit{Estimate for $\mathcal{J}_N^2$.} We get from Proposition \ref{L2bilin} that
\begin{align*}
  |\mathcal{J}_N^2(t)| &\lesssim   \sum_{N_1\gtrsim N} N^{2s} |G^2_t(u_{N_1}, u_{\sim N_1}, P_N^2\partial_x u)|\\
  &\lesssim  \sum_{N_1\gtrsim N} N^{2s-s'+\frac 12-\frac\alpha 4} N_1^{-2s+(1-\alpha)_+} \|u_{N_1}\|_{Y^s} \|u_{\sim N_1}\|_{Y^s} \|P_Nu\|_{Y^{s'}}.
\end{align*}
Since $ s>\frac1  2 -\frac \alpha 2 $  and $s'>s_\alpha$, we deduce that
\begin{equation}\label{prop-ee.4}
 \sup_{t\in ]0,T[} \sum_{N>N_0} |\mathcal{J}_N^2(t)| \lesssim  \sum_{N>N_0} N^{(s_\alpha-s')_+} \|u\|_{Y^{s'}} \|u\|_{Y^s}^2 \lesssim  N_0^{(s_\alpha-s')_+} \|u\|_{Y^{s'}} \|u\|_{Y^s}^2.
\end{equation}
It remains to estimate $\mathcal{J}_N^1+c\mathcal{K}_N$.  Using equation \eqref{dpB} we obtain
\begin{align*}
  \mathcal{K}_N(t) &= -N^{2s}\int_{\R^2_t} \frac{\chi_1(\xi_1,\xi_2)}{\Omega_2(\xi_1,\xi_2)} i\xi_1(\omega_{\alpha+1}(\xi_1)+\omega_{\alpha+1}(\xi_2)-\omega_{\alpha+1}(\xi)) \widehat{u_{\ll N}}(\xi_1) \widehat{u_{\sim N}}(\xi_2) \widehat{u_{\sim N}}(-\xi)\\
  &\quad + N^{2s}\int_{\R^2_t} \frac{\chi_1(\xi_1,\xi_2)}{\Omega_2(\xi_1,\xi_2)} \xi_1 \widehat{P_{\ll N}\partial_x(u^2)}(\xi_1) \widehat{u_{\sim N}}(\xi_2) \widehat{u_{\sim N}}(-\xi)\\
  &\quad + N^{2s}\int_{\R^2_t} \frac{\chi_1(\xi_1,\xi_2)}{\Omega_2(\xi_1,\xi_2)} \xi_1 \widehat{u_{\ll N}}(\xi_1) \widehat{P_{\sim N}\partial_x(u^2)}(\xi_2) \widehat{u_{\sim N}}(-\xi)\\
  &\quad + N^{2s}\int_{\R^2_t} \frac{\chi_1(\xi_1,\xi_2)}{\Omega_2(\xi_1,\xi_2)} \xi_1 \widehat{u_{\ll N}}(\xi_1) \widehat{u_{\sim N}}(\xi_2) \widehat{P_{\sim N}\partial_x(u^2)}(-\xi)\\
  &:= -\mathcal{J}_N^1 + \mathcal{K}_N^1+\mathcal{K}_N^2+\mathcal{K}_N^3.
\end{align*}
Taking $c=1$ this leads to estimate $\mathcal{J}_N^1+\mathcal{K}_N = \mathcal{K}_N^1+\mathcal{K}_N^2+\mathcal{K}_N^3$.

\medskip

\noindent \textit{Estimate for $\mathcal{K}_N^1$.} We have
$$
\mathcal{K}_N^1(t) = N^{2s} \sum_{N_1\ll N}\sum_{N_2\vee N_3\gtrsim N_1} (N_1N^\alpha)^{-1}N_1^2 \int_{\R_t} \Pi^3_{\chi_{\mathcal{K}^1}}(u_{N_2}, u_{N_3}, u_{\sim N}) u_{\sim N}
$$
where
$$
\chi_{\mathcal{K}^1}(\xi_1,\xi_2,\xi_3) = i \chi_1(\xi_1+\xi_2,\xi_3) \frac{N_1N^\alpha}{\Omega_2(\xi_1+\xi_2,\xi_3)} \frac{(\xi_1+\xi_2)^2}{N_1^2} \phi_{N_1}(\xi_1+\xi_2),
$$
with $\chi_1$ defined in \eqref{chi1}.
From Lemma \ref{omega.inv}, $\chi_{\mathcal{K}^1}$ satisfies \eqref{pseudoprod.1}. Therefore we get from Proposition \ref{eL2trilin} that
\begin{align*}
 |\mathcal{K}_N^1(t)| &\lesssim  \sum_{N_1\ll N} \sum_{N_2\vee N_3\gtrsim N_1}N_1 N^{-\alpha_+} N_2^{(\frac12)_-}\cro{N_2}^{-\frac \alpha 4-s'} N_3^{(\frac12)_-}\cro{N_3}^{-\frac \alpha 4-s'}  (N_2\vee N_3)^{0_+}\\
 & \hspace*{40mm} \times \|u_{N_2}\|_{Y^{s'}} \|u_{N_3}\|_{Y^{s'}} \|u_{\sim N}\|_{Y^s}^2\\
 & \lesssim   N^{-\alpha_+}\sum_{N_1\ll N}  N_1 \cro{N_1}^{\frac 1 2 -\frac \alpha 4 -s'_+} \|u\|_{Y^{s'}}^2 \|u\|_{Y^s}^2
\end{align*}
where in the last step we used that $ \frac 1 2 -\frac \alpha 4 -s' <0 $.
We thus infer that
\begin{equation}\label{prop-ee.7}
\sup_{t\in ]0,T[}\sum_{N>N_0} |\mathcal{K}_N^1(t)| \lesssim   (N_0^{-\alpha_+}+ N_0^{(s_\alpha-s')_+})\|u\|_{Y^{s'}}^2 \|u\|_{Y^s}^2.
\end{equation}
\medskip

\noindent \textit{Estimates for $\mathcal{K}_N^2+\mathcal{K}_N^3$.} Using, as in Subsection \ref{section-ee},  that
\begin{equation}\label{prop-ee.8}
P_{\sim N}(u^2) = 2u_{\ll N}u_{\sim N} + 2N^{-1}\Pi_{\tilde{\chi}}^2(\partial_xu_{\ll N}, u_{\sim N}) + P_{\sim N}(u_{\gtrsim N}u_{\gtrsim N}),
\end{equation}
where $\tilde{\chi}$ satisfies \eqref{pseudoprod.1}, we decompose $\mathcal{K}_N^2+\mathcal{K}_N^3$ as $\mathcal{K}_N^{31}+\mathcal{K}_N^{32}+\mathcal{K}_N^{33}$ with
$$
\mathcal{K}_N^{31}(t) = 2N^{2s} \int_{\R^2_t} \frac{\chi_1(\xi_1,\xi_2)}{\Omega_2(\xi_1,\xi_2)} \xi_1 \widehat{u_{\ll N}}(\xi_1) \left[ \widehat{\partial_x(u_{\ll N}u_{\sim N})}(\xi_2) \widehat{u_{\sim N}}(-\xi) + \widehat{u_{\sim N}}(\xi_2) \widehat{\partial_x(u_{\ll N}u_{\sim N})}(-\xi)\right],
$$
\begin{multline*}
\mathcal{K}_N^{32}(t) = 2N^{2s-1} \int_{\R^2_t} \frac{\chi_1(\xi_1,\xi_2)}{\Omega_2(\xi_1,\xi_2)} \xi_1 \widehat{u_{\ll N}}(\xi_1) \big[ \mathcal{F}_x(\partial_x\Pi_{\tilde{\chi}}^2(\partial_xu_{\ll N}, u_{\sim N}))(\xi_2) \widehat{u_{\sim N}}(-\xi) \\ + \widehat{u_{\sim N}}(\xi_2) \mathcal{F}(\partial_x\Pi_{\tilde{\chi}}^2(\partial_xu_{\ll N}, u_{\sim N}))(-\xi)\big],
\end{multline*}
and
\begin{multline*}
\mathcal{K}_N^{33}(t) = N^{2s} \int_{\R^2_t} \frac{\chi_1(\xi_1,\xi_2)}{\Omega_2(\xi_1,\xi_2)} \xi_1 \widehat{u_{\ll N}}(\xi_1) \big[ \mathcal{F}_x(\partial_xP_{\sim N}(u_{\gtrsim N}u_{\gtrsim N}))(\xi_2) \widehat{u_{\sim N}}(-\xi) \\ + \widehat{u_{\sim N}}(\xi_2) \mathcal{F}(\partial_xP_{\sim N}(u_{\gtrsim N}u_{\gtrsim N}))(-\xi)\big].
\end{multline*}

\medskip

\noindent \textit{Estimate for $\mathcal{K}_N^{31}$.} We have
\begin{align*}
 \mathcal{K}_N^{31}(t) &= -2iN^{2s} \int_{\R^3_t} \left(\frac{\chi_1}{\Omega_2}\right)(\xi_1,\xi_2) \widehat{\partial_x u_{\ll N}}(\xi_1) (i\xi_2) \widehat{u_{\ll N}}(\xi_3) \widehat{u_{\sim N}}(\xi_2-\xi_3) \widehat{u_{\sim N}}(-\xi_1-\xi_2)\\
 &\quad -2iN^{2s} \int_{\R^3_t} \left(\frac{\chi_1}{\Omega_2}\right)(\xi_1,\xi_2) \widehat{\partial_x u_{\ll N}}(\xi_1) \widehat{u_{\sim N}}(\xi_2) (-i(\xi_1+\xi_2)) \widehat{u_{\ll N}}(\xi_3) \widehat{u_{\sim N}}(-(\xi_1+\xi_2+\xi_3)).
\end{align*}
Now a change a variables leads to
$$
\mathcal{K}_N^{31} (t)= 2N^{2s} \sum_{N_1,N_2\ll N} \int_{\R^3_t} \sigma(\xi_1,\xi_2,\xi_3) \widehat{\partial_x u_{N_1}}(\xi_1) \widehat{u_{N_2}}(\xi_2) \widehat{u_{\sim N}}(\xi_3) \widehat{u_{\sim N}}(-\xi)
$$
with $\xi=\xi_1+\xi_2+\xi_3$ and
$$
\sigma(\xi_1,\xi_2,\xi_3) = \left(\frac{\chi_1}{\Omega_2}\right) (\xi_1,\xi_2+\xi_3)(\xi_2+\xi_3) - \left(\frac{\chi_1}{\Omega_2}\right)(\xi_1,\xi_3) (\xi_1+\xi_3).
$$
Let us rewrite $\sigma$ as follows:
\begin{align*}
  \sigma(\xi_1,\xi_2,\xi_3) &= \left[ \left(\frac{\chi_1}{\Omega_2}\right)(\xi_1,\xi_1+\xi_3) - \left(\frac{\chi_1}{\Omega_2}\right)(\xi_1,\xi_3) \right] \widetilde{\phi}_N(\xi_3)\xi_3\\
  &+ \left(\frac{\chi_1}{\Omega_2}\right)(\xi_1, \xi_2+\xi_3) \widetilde{\phi}_{N_2}(\xi_2)\xi_2\\
  &- \left(\frac{\chi_1}{\Omega_2}\right)(\xi_1,\xi_3) \widetilde{\phi}_{N_1}(\xi_1)\xi_1.
\end{align*}
According to Lemma \ref{omega.inv} and Remark \ref{symbolprod}, it is easy to check that $N_1N^\alpha \left(\frac{\chi_1}{\Omega_2}\right)(\xi_1, \xi_2+\xi_3)$, $\frac{1}{N_2} \widetilde{\phi}_{N_2}(\xi_2)\xi_2$ and thus $\frac{N_1N^\alpha}{N_2} \left(\frac{\chi_1}{\Omega_2}\right)(\xi_1, \xi_2+\xi_3) \widetilde{\phi}_{N_2}(\xi_2)\xi_2$ satisfy \eqref{pseudoprod.1}.
In the same way, $\frac{N_1N^\alpha}{N_1} \left(\frac{\chi_1}{\Omega_2}\right)(\xi_1,\xi_3) \widetilde{\phi}_{N_1}(\xi_1)\xi_1$ satisfies \eqref{pseudoprod.1}. Now we get from
the mean value theorem that for any multi-indice $\beta=(\beta_1, 0, \beta_3)$, there exists $|\tilde{\xi}_\beta|\sim N$  such that
\begin{align*}
  \partial^\beta \left[ \left(\frac{\chi_1}{\Omega_2}\right)(\xi_1,\xi_1+\xi_3) - \left(\frac{\chi_1}{\Omega_2}\right)(\xi_1,\xi_3) \right]
  &= \partial^{(\beta_1,\beta_3)}\left(\frac{\chi_1}{\Omega_2}\right)(\xi_1,\xi_2+\xi_3) - \partial^{(\beta_1\beta_3)}\left(\frac{\chi_1}{\Omega_2}\right)(\xi_1,\xi_3)\\
  &= \partial^{(\beta_1, \beta_3+1)} \left(\frac{\chi_1}{\Omega_2}\right)(\xi_1,\xi_\beta) \xi_2.
\end{align*}
On the other hand, for any $\beta=(\beta_1,\beta_2,\beta_3)$ with $\beta_2\ge 1$, we have
$$
\partial^\beta \left[ \left(\frac{\chi_1}{\Omega_2}\right)(\xi_1,\xi_1+\xi_3) - \left(\frac{\chi_1}{\Omega_2}\right)(\xi_1,\xi_3) \right] = \partial^{(\beta_1, \beta_2+\beta_3)}
\left(\frac{\chi_1}{\Omega_2}\right) (\xi_1, \xi_2+\xi_3).
$$
It thus follows from Lemma \ref{omega.inv} that $\frac{N_1N^{\alpha+1}}{N_2} \left[ \left(\frac{\chi_1}{\Omega_2}\right)(\xi_1,\xi_1+\xi_3) - \left(\frac{\chi_1}{\Omega_2}\right)(\xi_1,\xi_3) \right]$ satisfies \eqref{pseudoprod.1}. Therefore we deduce that $\chi_{\mathcal{K}^{31}} := \frac{N_1}{N_1\vee N_2}N^{\alpha}\sigma$ satisfies \eqref{pseudoprod.1}.
Rewriting $\mathcal{K}_N^{31}$ as
$$
\mathcal{K}_N^{31} = 2N^{2s-\alpha} \sum_{N_1,N_2\ll N} \frac{N_1\vee N_2}{N_1} \int_{\R_t} \Pi_{\chi_{\mathcal{K}^{31}}}^3( \partial_xu_{N_1}, u_{N_2}, u_{\sim N}) u_{\sim N}
$$
we get from estimate \eqref{eL2trilin.2} that
\begin{align*}
  |\mathcal{K}_N^{31}(t)| &\lesssim \, N^{-\alpha} \sum_{N_1,N_2\ll N} (N_1\vee N_2) N_1^{\frac12_-}\cro{N_1}^{-\frac \alpha 4-s'} N_2^{{\frac12}_-}\cro{N_2}^{-\frac \alpha 4-s'} N^{0_+} \|u_{N_1}\|_{Y^{s'}} \|u_{N_2}\|_{Y^{s'}} \|u_{\sim N}\|_{Y^s}^2\, .
\end{align*}
Recalling that $\frac 12-\frac \alpha 4-s'<0 $, it follows as in \eqref{prop-ee.7} that
\begin{equation}\label{prop-ee.10}
 \sup_{t\in ]0,T[}  \sum_{N>N_0} |\mathcal{K}_N^{31}(t)| \lesssim    (N_0^{-\alpha_+}+ N_0^{(s_\alpha-s')_+})\|u\|_{Y^{s'}_T}^2\|u\|_{Y^s_T}^2.
\end{equation}

\medskip

\noindent \textit{Estimate for $\mathcal{K}_N^{32}$.} We only deal with the first term $\mathcal{K}_N^{321}$ of the sum in $\mathcal{K}_N^{32}$ since the other is estimated similarly. With the notation of Section \ref{pseudoprodest} we obtain
$$
\mathcal{K}_N^{321}(t) = 2N^{2s}\sum_{N_1,N_2\ll N} N_1N_2 (N_1N^\alpha)^{-1} \int_{\R_t} \Pi^3_{\chi_{\mathcal{K}^{321}}}( u_{N_1}, u_{N_2}, u_{\sim N}) u_{\sim N}
$$
with
$$
\chi_{\mathcal{K}^{321}}(\xi_1,\xi_2,\xi_3) = -\chi_1(\xi_1,\xi_2+\xi_3)\tilde{\chi}(\xi_2,\xi_3) \frac{N_1N^\alpha}{\Omega_2(\xi_1,\xi_2+\xi_3)} \frac{\xi_1}{N_1} \frac{\xi_2}{N_2} \frac{\xi_2+\xi_3}{N}.
$$
Noticing that $\chi_{\mathcal{K}^{321}}(\xi_1,\xi_2)$ satisfies condition \eqref{pseudoprod.1}, estimate \eqref{eL2trilin.2} implies that
\begin{align*}
 |\mathcal{K}_N^{321}(t)| &\lesssim \,N^{-\alpha} \sum_{N_1,N_2\ll N} N_2 N_1^{(\frac12)_-} \cro{N_1}^{-\frac \alpha 4-s'} N_2^{(\frac12)_-} \cro{N_2}^{-\frac \alpha 4-s'} N^{0_+} \|u_{N_1}\|_{Y^{s'}} \|u_{N_2}\|_{Y^{s'}} \|u_{\sim N}\|_{Y^s}^2\, .\\
\end{align*}
which again, as in  \eqref{prop-ee.7}, leads to
\begin{equation}\label{prop-ee.11}
  \sup_{t\in ]0,T[} \sum_{N>N_0} |\mathcal{K}_N^{32}(t)| \lesssim  (N_0^{-\alpha_+}+ N_0^{(s_\alpha-s')_+}) \|u\|_{Y^{s'}_T}^2 \|u\|_{Y^s_T}^2.
\end{equation}

\medskip

\noindent \textit{Estimate for $\mathcal{K}_N^{33}$.} We follow again the same arguments. We only deal with the first term $\mathcal{K}_N^{331}$ of the sum in $K_N^{33}$ and rewrite it as
$$
\mathcal{K}_N^{331}(t) = N^{2s} \sum_{N_1\ll N}\sum_{N_2\gtrsim N}  N_1 (N_1N^\alpha)^{-1} N \int_{\R_t} \Pi^3_{\chi_{\mathcal{K}^{331}}} (u_{N_1}, u_{N_2}, u_{\sim N_2}) u_{\sim N}
$$
with
$$
\chi_{\mathcal{K}^{331}}(\xi_1,\xi_2,\xi_3) = i \chi_1(\xi_1,\xi_2+\xi_3) \frac{N_1N^\alpha}{\Omega_2(\xi_1,\xi_2+\xi_3)} \frac{\xi_1}{N_1} \frac{\xi_2+\xi_3}N \phi_{\sim N}(\xi_2+\xi_3).
$$
Then, thanks to estimate \eqref{eL2trilin.2}, we get
\begin{align*}
 |\mathcal{K}_N^{331}(t)| &\lesssim \, N^{(2s-s'+s_\alpha)} \sum_{N_1\ll N} \sum_{N_2\gtrsim N} N_1^{(\frac12)_-} \cro{N_1}^{-\frac \alpha 4-s'}  N_2^{-2s_+} \|u_{N_1}\|_{Y^{s'}} \|u_{N_2}\|_{Y^s} \|u_{\sim N_2}\|_{Y^s} \|u_{\sim N}\|_{Y^{s'}}\\
 &\lesssim  N^{(2s-s'+s_\alpha)}  \sum_{N_2\gtrsim N} N_2^{-2s_+} \|u\|_{Y^{s'}}^2 \|u_{N_2}\|_{Y^s}^2.
\end{align*}
This leads to
\begin{equation}\label{prop-ee.12}
 \sup_{t\in ]0,T[}  \sum_{N>N_0} |\mathcal{K}_N^{33}(t)| \lesssim   N_0^{(s_\alpha-s')_+}  \|u\|_{Y^{s'}}^2 \|u\|_{Y^s}^2.
\end{equation}
Combining \eqref{prop-ee.2}-\eqref{prop-ee.3}-\eqref{prop-ee.4}-\eqref{prop-ee.7}-\eqref{prop-ee.10}-\eqref{prop-ee.11}-\eqref{prop-ee.12}, we conclude the proof of Proposition \ref{prop-ee}.
\end{proof}
\begin{corollary} \label{coro-ee}
Let $0<\alpha \le 1$.  Let $s>s_\alpha$, $0<T \le 1$ and $u\in Y^s_T$ be a solution of \eqref{dpB} on $[0,T]$. Then for any $ N_0\gg 1 $ we have
\begin{equation} \label{prop-ee.13} 
\sup_{t\in ]0,T[}\sum_{N>N_0}\cro{N}^{2s}\Bigl| \mathcal{E}_N(u(t),N_0)-  \mathcal{E}_N(u_0,N_0) \Bigr| \lesssim ( N_0^{(s_\alpha-s)_+}+N_0^{-\alpha_+} )(\|u\|_{Y^s_T}^3 + \|u\|_{Y^s_T}^4) \, .
\end{equation}
\end{corollary}
\begin{proof}
According to \eqref{prop-ee.3}, it suffices to bound
$$
\sup_{t\in ]0,T[} \sum_{N>N_0} \Bigl|\mathcal{J}_N(t) +  \mathcal{K}_N(t)\Bigr|
$$
and the result follows from by combining \eqref{prop-ee.4}-\eqref{prop-ee.7}-\eqref{prop-ee.10}-\eqref{prop-ee.11}-\eqref{prop-ee.12}.
\end{proof}

\section{Estimates for the difference of two solutions}

In this section, we provide the needed estimates for the difference $w$ of two solutions $u,v$ of \eqref{dpB}. If $w=u-v$ and $z=u+v$, then
\begin{equation}\label{eqdiff}
(\partial_t-L_{\alpha+1})w=\partial_x(zw) \, .
\end{equation}
The lack of symmetry in the nonlinear term of \eqref{eqdiff} prevents us to estimate $w$ in $Y^s_T$, $s>s_\alpha$. To overcome this difficulty, we will rather work at a lower regularity level $\sigma<0$ and more precisely with
\begin{equation}\label{sigma}
 \sigma\in \Bigl]-\frac 1 2 +\frac \alpha 4,\min(0,s-2+\frac{3}{2} \alpha) \Bigr[ \; .
 \end{equation}
 \begin{remark}  \label{rk51}

For $ \alpha\in ]0,1] $ and $ s>s_\alpha=\frac 3 2 - \frac 5 4 \alpha $, it holds  $-\frac 1 2 +\frac \alpha 4<0 $ and  $s-2+\frac{3}{2} \alpha>-\frac 1 2 + \frac \alpha 4 $. Therefore, the definition interval in \eqref{sigma} is never empty. Moreover, it is worth noticing that $ -\sigma<\frac 1 2 -\frac \alpha 4\le s_\alpha<s $.

\end{remark}
  Since we are not able to control the $X_T^{\sigma-1,1}\cap L^2_TW^{1-\alpha+(\sigma-s_\alpha)_-, \infty}_x$ part of $w$ for $\sigma<0$ we need to bound the difference in the sum space $F^{\sigma, \frac12}$. Finally, to treat some low-high interactions in the energy estimates, we also need to add a weight on the low space frequencies so that $w$ will take place in $\overline{Z}^\sigma$.

\subsection{Bilinear estimate}

\begin{proposition}\label{bediff} Let $0<\alpha \le 1$.  Assume that $0<T\le 1$,  $s>s_\alpha$ and  $ -\frac 1 2 +\frac \alpha 4< \sigma<\min (0,s-2+\frac{3}{2} \alpha) $.
Let $z\in Y^s_T$ and let $w\in \overline{Z}^\sigma_T$  be a solution of \eqref{eqdiff} on $]0,T[$ with  $ w_0\in \overline{H^\sigma} $. Then  it holds
 \begin{equation}\label{bediff0}
    \|w\|_{\overline{F}_T^{\sigma, 1/2}} \lesssim \|w_0\|_{\overline{H}^{\sigma}} +(1+\|z\|_{Y^s_T})\|z\|_{Y^s_T} \|w\|_{L^\infty_T H^\sigma_x}
   \, .
  \end{equation}
\end{proposition}

\begin{proof}
Let $ \tilde{w} =\rho_T(w)$  and $  \tilde{z} =\rho_T(z)$ be the extensions defined in \eqref{defrho} and let $\tilde{\tilde{w}}$ satisfying   \eqref{eqdiff} with $\partial_x(\tilde{z}\tilde{w})$ as second hand member. 
We will estimate the extension $\breve{w}=\eta \tilde{\tilde{w}} $ of $ w $ where $ \eta $ is the smooth cut-off function defined in \eqref{eta}. 
To simplify the notation we drop the tilde in the sequel.
 For $ N_0\ge 1 $ to be chosen later, we rewrite $ zw $ as
\begin{eqnarray}
zw & = & P_{\le N_0} (zw) +\sum_{N>N_0} \Bigl[ z_{\lesssim N} w_{\sim N} + z_{\sim N} w_{\lesssim N} +\sum_{N_1\gg N} 
 z_{N_1} w_{\sim N_1} \Bigl]\nonumber  \\
 & =: &  J_{\le N_0}+J^{l,h}_{>N_0}+J^{h,l}_{>N_0}+ J_{>N_0}^{h,h}\; .\label{decomp}
\end{eqnarray}
Duhamel formula,  \eqref{Fprop}, as well as classical Bourgain's estimate on the linear evolution  (cf. \cite{Bo1}, \cite{G}) and \eqref{Fprop} lead to
\begin{align*}
  \|\breve{w}\|_{\overline{F}^{\sigma, \frac12}} & \lesssim \|w_0\|_{\overline{H}^{\sigma}} + \| \partial_x(zw)\|_{\overline{F}^{\sigma, -\frac12}} \\
  &  \lesssim \|w_0\|_{\overline{H}^{\sigma}} + \|J_{\le N_0}\|_{X^{\sigma,0}}+
  \|J^{l,h}_{>N_0}+J^{h,l}_{>N_0} \|_{X^{\sigma,0}} + \|J_{>N_0}^{h,h}\|_{F^{\sigma+1, -\frac12}}
\end{align*}
    Now, using that $0<-\sigma<s $,  we easily bound the contribution  of the low frequency part $ J_{\le N_0} $  by
 \begin{equation}\label{rr1}
    \|J_{\le N_0} \|_{X^{\sigma,0}} \lesssim  \sum_{N\le N_0} N^{\frac12}\cro{N}^{\sigma}  \| P_N (zw)\|_{L^\infty_t L^1_x}
    \lesssim  N_0^{\frac12}\|z\|_{L^\infty_T H^s_x} \|w\|_{L^\infty_T H^{\sigma}_x}\ .
    \end{equation}

 The contribution of the high-low interactions $ J_{>N_0}^{h,l} $ is also easily bounded as follows
\begin{align}
\|J_{> N_0}^{h,l}  \|_{X^{\sigma,0}}  &\lesssim \sum_{N>N_0} \cro{N}^{\sigma} \|z_{\sim N}\|_{L^2_tL^\infty_x} \|w_{\lesssim N}\|_{L^\infty_tL^2_x}\nonumber \\
    &\lesssim \sum_{N>N_0} \|z_{N}\|_{L^2_t L^\infty_x} \|w\|_{L^\infty_t H^\sigma_x} \nonumber\\
    & \lesssim  \|z\|_{Y^s} \|w\|_{L^\infty_t H^\sigma_x} \; ,\label{rr2}
\end{align}
  where in the next to the  last step we used that $ \sigma<0 $ yields $ \cro{N}^{\sigma}  \|w_{\lesssim N}\|_{L^\infty_tL^2_x}\lesssim  \|w\|_{L^\infty_t H^\sigma_x} $.
    To bound the contribution of the low-high  interactions $ J_{>N_0}^{l,h} $ we write
  \begin{align}
  \|J_{> N_0}^{l,h}  \|_{X^{\sigma,0}}
     &\lesssim  \Bigl\| \Bigl(\sum_{N>N_0} \cro{N}^{2\sigma} \|z_{\lesssim N} w_{\sim N}\|_{L^2_x}^2\Bigr)^{1/2}\Bigr\|_{L^2_t}\nonumber\\
    &\lesssim  \Bigl\| \Bigl(\sum_{N>N_0} \cro{N}^{2\sigma} \|z_{\lesssim N}\|_{L^\infty_x}^2\|w_{\sim N}\|_{L^2_x}^2\Bigr)^{1/2}\Bigr\|_{L^2_t}\nonumber\\
    &\lesssim \Bigl\| \|z\|_{L^\infty_x } \|w\|_{H^{\sigma}_x}\Bigr\|_{L^2_t} \nonumber\\
    & \lesssim \|z\|_{L^2_T L^\infty_x} \|w\|_{L^\infty_T H^{\sigma}_x}\, . \label{rr3}
  \end{align}
  Now we deal with the (high-high) interactions term 
  $$
  \|J_{>N_0}^{h,h}\|_{F^{\sigma+1, -\frac12}} \lesssim \sum_{N>N_0 \atop N_1 \gg N} \left\| \sum_{L_{max}\gtrsim NN_1^{\alpha}} P_NQ_L(Q_{L_1}z_{N_1} Q_{L_2}w_{\sim N_1})\right\|_{F^{\sigma+1, -\frac12}}.
  $$  
 To estimate the contribution of  the sum over $L\gtrsim NN_1^\alpha$, we take advantage of the $X^{\sigma+1,(-\frac 12)_+}$-part of $F^{\sigma+1,-\frac 12}$. Therefore this term is bounded by
 \begin{align}
   &\sum_{N>N_0 \atop N_1 \gg N} \sum_{L\gtrsim NN_1^{\alpha}} \|P_NQ_L(z_{N_1} w_{\sim N_1})\|_{X^{\sigma+1,(-\frac 12)_+}} \nonumber \\
      &\quad\lesssim \sum_{N>N_0 \atop N_1 \gg N} \sum_{L\gtrsim NN_1^\alpha} N^{1+\sigma} L^{(-\frac12)_+} \|P_NQ_L(z_{N_1} w_{\sim N_1})\|_{L^2_{tx}}\nonumber\\
    &\quad \lesssim  \sum_{N>N_0} N^{\sigma+(\frac12)_+} \sum_{N_1\gg N}  N_1^{(-\frac \alpha 2)_+}
   \|z_{N_1}\|_{L^2_tL^\infty_x} \|w_{\sim N_1}\|_{L^\infty_t L^2_x}\nonumber\\
      &\quad \lesssim  \sum_{N>N_0} N^{\sigma+(\frac12)_+} \sum_{N_1\gg N}  N_1^{(-\frac \alpha 2-\sigma)_+}
    N_1^{(s_\alpha-s)+(\alpha-1)_+} \|z_{N_1}\|_{Y^s} \|w_{\sim N_1}\|_{L^\infty_t H^\sigma_x}\nonumber\\
    &\quad \lesssim  \sum_{N>N_0} N^{(s_\alpha-s)+\frac 1 2 (\alpha-1)_+} \|z\|_{Y^s} \|w\|_{L^\infty_tH^\sigma_x}\nonumber \\
     & \quad \lesssim \|z\|_{Y^s} \|w\|_{L^\infty_tH^\sigma_x} \; ,\label{rr4}
  \end{align}
  where we used that $-\frac \alpha 2-\sigma+(s_\alpha-s)+(\alpha-1)_+\le  -\sigma-(1-\frac \alpha 2) <0 $  since
   $ 0<-\sigma< \frac 1  2 -\frac \alpha 4 $.  The contribution of the region $L\ll NN_1^\alpha$ and $L_1\gtrsim NN_1^\alpha$ is estimated by
  \begin{align}
  \sum_{N>N_0 \atop N_1 \gg N} \|Q_{\ll N N_1^\alpha}& (Q_{\gtrsim  N N_1^\alpha}z_{N_1} w_{\sim N_1})\|_{X^{\sigma,0}}  \nonumber   \\ &\quad \lesssim
   \sum_{N>N_0}\sum_{N_1\gg N} N^{\sigma+\frac12} \|Q_{\gtrsim NN_1^\alpha}z_{N_1} w_{\sim N_1}\|_{L^2_tL^1_x}\nonumber\\
    &\quad \lesssim\sum_{N>N_0} \sum_{N_1\gg N} N^{\sigma+\frac12} (NN_1^\alpha)^{-1} \|z_{N_1}\|_{X^{0,1}} \|w_{\sim N_1}\|_{L^\infty_tL^2_x}\nonumber\\
    &\quad \lesssim \sum_{N>N_0} N^{\sigma-\frac12} \sum_{N_1\gg N}N_1^{1-\alpha-s-\sigma} \|z\|_{X^{s-1,1}} \|w\|_{L^\infty_tH^\sigma_x} \nonumber\\
     &\quad \lesssim \sum_{N>N_0} N^{\frac 1 2 -\alpha -s }\|z\|_{X^{s-1,1}} \|w\|_{L^\infty_tH^\sigma_x}\nonumber \\
        &\quad  \lesssim N_0^{-1+\frac \alpha 4 }\|z\|_{X^{s-1,1}} \|w\|_{L^\infty_tH^\sigma_x},\label{rr5}
  \end{align}
  where we used that $-\alpha-s+1-\sigma<s_\alpha-s<0 $ to sum over $ N_1 $.
   Finally the contribution of the last region can be bounded thanks to Lemmas \ref{QLbound} and \ref{eqFsb} by
  \begin{align}
   \sum_{N>N_0 \atop N_1 \gg N} \|Q_{\ll N N_1^\alpha} &  (Q_{\ll  N N_1^\alpha}z_{N_1}  \; Q_{\gtrsim N N_1^\alpha}   w_{\sim N_1})\|_{X^{\sigma,0}} \nonumber  \\
    &\quad \lesssim \sum_{N>N_0 \atop N_1 \gg N} N^{\sigma+\frac12} \|Q_{\ll NN_1^\alpha}z_{N_1}  Q_{\gtrsim NN_1^\alpha}w_{\sim N_1}\|_{L^2_tL^1_x}\nonumber\\
    &\quad \lesssim \sum_{N>N_0 \atop N_1 \gg N} N^{\sigma+\frac12} (NN_1^\alpha)^{-1}N_1 \|z_{N_1}\|_{L^\infty_tL^2_x} \|w_{\sim N_1}\|_{F^{0,\frac12}}\nonumber\\
    &\quad \lesssim \sum_{N>N_0} N^{\sigma-\frac12} \sum_{N_1\gg N}N_1^{1-\alpha-s-\sigma}\|z\|_{L^\infty_tH^s_x} \|w\|_{F^{ \sigma,\frac12}} \nonumber\\
     &\quad \lesssim \sum_{N>N_0} N^{\frac 1 2 -\alpha -s }\|z\|_{L^\infty_t H^s_x} \|w\|_{F^{\sigma,\frac12}} \nonumber\\
        &\quad  \lesssim N_0^{-1+\frac \alpha 4 }\|z\|_{L^\infty_t H^s_x} \|w\|_{F^{\sigma,\frac12}},\label{rr6}
  \end{align}
  where we used that $(NN_1^{\alpha})^{-\frac12} \le (NN_1^{\alpha})^{-1}N_1$, since $\alpha \le 1$.
  Gathering \eqref{rr1}-\eqref{rr6} we obtain
  $$
     \|\breve{w}\|_{\overline{F}^{\sigma, \frac12}} \le  c_1 \|w_0\|_{\overline{H}^{\sigma}} +c_2 (N_0^{\frac12} +1) \|z\|_{Y^s_T} \|w\|_{L^\infty_T H^\sigma_x}
   +  c_3  N_0^{-\frac12}   \|z\|_{Y^s_T} \|w\|_{F_T^{\sigma, \frac12}}\, ,
$$
where $ c_1, \, c_2\,,  c_3 \ge 1 $. This yields the desired result by taking $ N_0=[2c_3 (1+\|z\|_{Y^s_T} )]^2$ and  concludes the proof of Proposition \ref{bediff}.
\end{proof}

\subsection{Refined Strichartz estimate}

\begin{proposition}\label{sediff}
Let $0<\alpha \le 1.$  Assume that $0<T\le 1$, $s>s_\alpha$ and  $ -\frac 1 2 +\frac \alpha 4< \sigma<\min (0,s-2+\frac{3}{2} \alpha) $. Let $z\in Y^s_T$ and $w\in \overline{Z}^\sigma_T$ be a solution of \eqref{eqdiff} on $]0,T[$. Then
\begin{equation}\label{sediff.0}
\|J^{(\sigma-s_\alpha)+(1-\alpha)_-}_x w\|_{L^2_T \overline{L^\infty_x}} \lesssim (1+ \|z\|_{Y^s_T} ) \|w\|_{L^\infty_T \overline{H}^\sigma_x}\;  .
\end{equation}
\end{proposition}

\begin{proof}
The low frequency part is estimated by
  \begin{align*}
  \|P_{\lesssim 1} J^{1-\alpha+(\sigma-s_\alpha)}_x w\|_{L^2_T\overline{L^\infty_x} } & \lesssim T^{\frac12} \sum_{N\lesssim 1}  \cro{N^{-1}}N^{\frac12} \|w_N\|_{L^\infty_TL^2_x}\\
  & \lesssim T^{\frac12} \|w\|_{L^\infty_T \overline{H}^\sigma_x}.
  \end{align*}
To estimate the high frequency part of  the LHS of \eqref{sediff.0}, we decompose $zw $ as in \eqref{decomp} and we use Lemma \ref{se} with $ \delta=1 $ to get
  \begin{align}
  N^{(\sigma-s_\alpha)+(1-\alpha)_-} \|w_N \|_{L^2_TL^\infty_x} &\lesssim  N^{0-} \|w_N\|_{L^\infty_TH^\sigma_x} +
   N^{\sigma_-} \|z_{\lesssim N} w_{\sim N} \|_{L^2_{T,x}} \nonumber \\
   & + N^{\frac{1}{2}-\frac{\alpha}{4}+\sigma_-}\Bigl(  \|z_{\sim N} w_{\lesssim N} \|_{L^2_T L^1_x} + \sum_{N_1\gg N} \| z_{N_1} w_{\sim N_1} \|_{L^2_T L^1_x}\Bigr)\nonumber \\
   & \lesssim  N^{0-} \|w_N\|_{L^\infty_TH^\sigma_x} +N^{0_-} \|z\|_{L^2_T L^\infty_x} \|w\|_{L^\infty_T H^{\sigma}_x}\nonumber \\
   & +N^{\frac{1}{2}-\frac{\alpha}{4}+\sigma_-} N^{-s-\sigma} \|z\|_{L^\infty_T H^s_x} \|w\|_{L^\infty_T H^{\sigma}_x} 
  \end{align}
  where we used that $\sigma<0 $ and $ s+\sigma>0$.  Summing over $ N\gg 1$,  using that $ s>s_\alpha\ge \frac{1}{2}-\frac{\alpha}{4} $, \eqref{sediff.0} follows.
\end{proof}
   
\begin{corollary}\label{coro2}
Let $0<\alpha \le 1.$ Assume that $0<T\le 1$, $s>s_\alpha$ and  $ -\frac 1 2 +\frac \alpha 4< \sigma<\min (0,s-2+\frac{3}{2} \alpha) $. Let $z\in Y^s_T$ and $w\in \overline{Z}^\sigma_T$ be a solution of \eqref{eqdiff} on $]0,T[$ such that $ w_0\in \overline{H}^\sigma $. Then
  \begin{equation}\label{co2}
    \|w\|_{\overline{Z}^{\sigma}_T} \lesssim (1+\|z\|_{Y^s_T} )^2\Bigl( \|w_0\|_{ \overline{H}^\sigma}+ \|w\|_{L^\infty_T \overline{H}^\sigma_x}\Bigr) \, .
  \end{equation}
\end{corollary}
\begin{proof} By the property of the extension $ \tilde{w}=\rho_T (w) $ defined in \eqref{defrho} we have
 \begin{equation}\label{coco2}
   \|\tilde{w}\|_{\overline{Z}^{\sigma}}\lesssim \|w\|_{L^\infty_T \overline{H}^\sigma_x}+\|w\|_{\overline{F}^{\sigma}_T}+ \|J^{(\sigma-s_\alpha)+(1-\alpha)_-}_x w\|_{L^2_T \overline{L^\infty_x}}
  \end{equation}
and the result follows by gathering this last estimate with \eqref{bediff0} and \eqref{sediff.0}.
\end{proof}
\subsection{Energy estimate}
For $ N_0 \ge  1 $, we define the modified energy for the diffe-rence $w$ of two solutions $ u$ and $ v $ by
\begin{equation} \label{defENt}
  \widetilde{\mathcal{E}}_N(z,w,N_0) = \left\{ \begin{array}{ll} \frac 12  \|P_Nw\|_{L^2_x}^2 & \text{for} \ N \le N_0 \, \\
\frac 12 \|P_Nw\|_{L^2_x}^2 + \widetilde{c_1} \widetilde{\mathcal{E}}_N^{1}(z,w) +  \widetilde{c_2} \widetilde{\mathcal{E}}_N^{2}(z,w) & \text{for} \ N> N_0 \, , \end{array}\right.
\end{equation}
where
$$
\widetilde{\mathcal{E}}_N^1(z,w) = \int_{\R^2} \left(\frac{\widetilde{\chi_1}}{\Omega_2}\right)(\xi_1,\xi_2) \xi_1 \widehat{z_{\ll N}}(\xi_1) \widehat{P_{\sim N}w}(\xi_2)\widehat{P_{\sim N}w}(-\xi_1-\xi_2) d\xi_1d\xi_2,
$$
and
$$
\widetilde{\mathcal{E}}_N^2(z,w) = \int_{\R^2} \left(\frac{\widetilde{\chi_2}}{\Omega_2}\right) (\xi_1,\xi_2) (\xi_1+\xi_2) \widehat{w_{\ll N}}(\xi_1) \widehat{P_{\sim N}z}(\xi_2)\widehat{P_{\sim N}w}(-\xi_1-\xi_2) d\xi_1d\xi_2 \, ,
$$
{$\Omega_2$ is defined in \eqref{Omega2}, $\tilde{\chi}_1$, $\tilde{\chi}_2$ are symbols satisfying the Marcinkiewicz condition \eqref{pseudoprod.1} and defined later in the proof of Proposition \ref{prop-eed}, and $\widetilde{c_1}$, $\widetilde{c_2}$ are real constants that will be fixed later in the proof of Proposition \ref{prop-eed}.

We define the modified energy  at the $ H^\sigma$-regularity associated with the difference of two solutions by using a homogeneous dyadic decomposition in spatial frequency
\begin{equation}\label{def-EsTt}
\widetilde{E}^{\sigma}(z,w,N_0) =  \sum_{N>0} \cro{N^{-1}}^2 \cro{N}^{2\sigma} \big|\widetilde{\mathcal{E}}_N(z,w,N_0)\big| \, .
\end{equation}

\begin{lemma}[Coercivity of the modified energy]\label{lem-EsTt}
Let $0<\alpha \le 1$, $s > s_\alpha$, $0<T \le 1$ and  $ -\frac 1 2 +\frac \alpha 4	< \sigma<\min (0,s-2+\frac{3}{2} \alpha) $.  Let $z\in Y^s_T$ and $w\in \overline{Z}^{\sigma}_T$ be a solution of \eqref{eqdiff}. Then  for $N_0 \gg   (1+\|z\|_{H^s_x})^\frac{2}{\alpha} $  it holds
\begin{equation} \label{lem-Estt.1}
\Bigl|  \widetilde{E}^\sigma(z,w,N_0) -  \frac{1}{2}  \sum_{N>0 } \cro{N}^{2\sigma} \cro{N^{-1}}^2 \|P_Nw\|_{L^2_x}^2
 \Bigr| \le \frac{1}{8}   \sum_{N>N_0 } \cro{N}^{2\sigma} \cro{N^{-1}}^2 \|P_Nw\|_{L^2_x}^2 \; .
 \end{equation}
\end{lemma}

\begin{proof}
We infer from (\ref{def-EsTt}) and the triangle inequality that, for $ N_0\gg 1$,
\begin{align}
\Bigl|  \widetilde{E}^\sigma(z,w,N_0) - & \frac{1}{2}  \sum_{N>0 } \cro{N}^{2\sigma} \cro{N^{-1}}^2 \|P_Nw\|_{L^2_x}^2
 \Bigr| \nonumber \\
 & \lesssim \sum_{N> N_0}N^{2\sigma} \big|\widetilde{\mathcal{E}}_N^{1}(z,w)\big|  + \sum_{N> N_0}N^{2\sigma} \big|\widetilde{\mathcal{E}}_N^{2}(z,w)\big|\, .\label{lem-Estt.2}
\end{align}
Thanks to Young and Bernstein's inequalities we have for $ N\ge N_0\gg 1  $,
\begin{equation} \label{lem-Estt.3}
\begin{split}
N^{2\sigma}\big|\widetilde{\mathcal{E}}_N^{1}(z,w)\big| &\lesssim \sum_{N_1\ll N} N^{2\sigma} (N_1N^\alpha)^{-1} N_1^{\frac12} \|\partial_xz_{N_1}\|_{L^2_x} \|w_{\sim N}\|_{L^2_x}^2 \\ &
\lesssim  (N^{-\alpha}+N^{\frac \alpha 4 -1})\|z\|_{H^{s}_x} \|w_{\sim N}\|_{H^\sigma_x}^2 \, .
\end{split}
\end{equation}
Similarly we bound the contribution of $\widetilde{\mathcal{E}}^2_N$ for $ N\ge  N_0\gg 1 $ by
\begin{align}
N^{2\sigma} \big|\widetilde{\mathcal{E}}^2_N(z,w)\big| &\lesssim \sum_{N_1\ll N} N^{\sigma-s+1-\alpha} N_1^{-\frac12} \cro{N_1^{-1}}^{-1} \cro{N_1}^{-\sigma} \nonumber\\
&\quad \times \|w_{N_1}\|_{\overline{H}^\sigma_x} \|z_{\sim N}\|_{H^s_x} \|w_{\sim N}\|_{H^\sigma_x} \nonumber\\
&\lesssim  N^{\frac \alpha 2 -1}\|w\|_{\overline{H}^\sigma_x} \|z_{\sim N}\|_{H^s_x} \|w_{\sim N}\|_{H^\sigma_x}. \label{lem-Estt.4}
\end{align}
Finally, we conclude the proof of \eqref{lem-Estt.1} gathering \eqref{lem-Estt.2}-\eqref{lem-Estt.3}-\eqref{lem-Estt.4} and the fact that
 $ \|w\|_{\overline{H}^\sigma_x}\sim\displaystyle  \sum_{N>0 } \cro{N}^{2\sigma} \cro{N^{-1}}^2 \|P_Nw\|_{L^2_x}^2$.
\end{proof}

\begin{proposition}\label{prop-eed}
 Let $0<\alpha\le 1$.  Let $s > s_\alpha$, $0<T \le 1$ and  $ -\frac 1 2 +\frac \alpha 4< \sigma<\min (0,s-2+\frac{3}{2} \alpha) $.  Let $u,v\in Y^s_T$ two solutions of \eqref{dpB} such that 
  $w=u-v\in \overline{Z}^{\sigma}_T$. Then, setting $z=u+v$, it holds
  \begin{eqnarray}\
\sup_{t\in ]0,T[}   \widetilde{E}^\sigma(z(t),w(t),N_0) &  \lesssim  & \widetilde{E}^\sigma(z(0),w(0),N_0)+ (T N_0^{\frac32} +N_0^{(s_\alpha-s)_+})\|z\|_{Y^s_T} \|w\|_{\overline{Z}^\sigma_T}^2  \nonumber \\
& & \hspace*{-15mm}+(N_0^{-(\frac{\alpha}2)+}+ N_0^{(s_\alpha-s)_+}+N_0^{-2\gamma+(\alpha-1)_+}) (\|u\|_{Y^s_T}^2+\|v\|_{Y^s_T}^2)\|w\|_{\overline{Z}^\sigma_T}^2 ,\label{prop-eed.0}
  \end{eqnarray}
  where $ \gamma=s-2+\frac{3}{2} \alpha-\sigma>0 $.
\end{proposition}

\begin{proof}
  We argue as in the proof of Proposition \ref{prop-ee}. To deal with the low frequencies $N\le N_0$, we use equation \eqref{eqdiff} to deduce
  $$
  \frac{d}{dt} \widetilde{\mathcal{E}}_N(z(t),w(t)) = \int_\R P_N\partial_x(zw) P_Nw
  $$
  for any $t\in (0,T]$. Integrating this on $(0,t)$ it follows after a dyadic decomposition of $P_N(zw)$ that
 \begin{displaymath}
  \begin{split}
    |\widetilde{\mathcal{E}}_N(z(t),w(t))| &\lesssim |\widetilde{\mathcal{E}}_N(z(0),w(0))| + N^{\frac32}T \|z_{\lesssim N}\|_{L^\infty_TL^2_x} \|w_{\lesssim N}\|_{L^\infty_TL^2_x} \|w_N\|_{L^\infty_TL^2_x}\\
   & \quad  +\sum_{N_1\gg N} N^{\frac32}  T \|z_{N_1}\|_{L^\infty_TL^2_x} \|w_{\sim N_1}\|_{L^\infty_TL^2_x} \|w_N\|_{L^\infty_TL^2_x}
     \\ &=:|\widetilde{\mathcal{E}}_N(z(0),w(0))|+ I_N+II_N \, .
     \end{split}
  \end{displaymath}
On the one hand, we infer
\begin{displaymath}
\cro{N^{-1}}^2\cro{N}^{2\sigma}I_N \lesssim TN^{\frac32}\|z\|_{L^{\infty}_TL^2_x}\|w\|_{L^{\infty}_T\overline{H}^{\sigma}_x}\|w_N\|_{L^{\infty}_T\overline{H}^{\sigma}_x} \, ,
\end{displaymath}
by using that $\cro{N^{-1}}\cro{N}^{\sigma}\|w_{\lesssim N}\|_{L^\infty_TL^2_x} \lesssim \|w\|_{L^{\infty}_T\overline{H}^{\sigma}_x} $. On the other hand, recalling that $0<-\sigma<s$, we get
\begin{displaymath}
\cro{N^{-1}}^2\cro{N}^{2\sigma}II_N \lesssim \cro{N}^{-s}\cro{N^{-1}}N^{\frac32} \|z\|_{L^{\infty}_TH^s_x}\|w\|_{L^{\infty}_TH^{\sigma}_x}\|w_N\|_{L^{\infty}_T\overline{H}^{\sigma}_x}\, .
\end{displaymath}
Therefore, we deduce by summing over $N \le N_0$ that
\begin{displaymath}
\begin{split}
\sum_{N\le N_0}  \cro{N^{-1}}^2& \cro{N}^{2\sigma} |\widetilde{\mathcal{E}}_N(z(t),w(t))| \\&\lesssim  \sum_{N\le N_0}  \cro{N^{-1}}^2\cro{N}^{2\sigma} |\widetilde{\mathcal{E}}_N(z(0),w(0))| + T N_0^{\frac32}\|z\|_{L^\infty_TH^s_x} \|w\|_{L^\infty_T \overline{H}^\sigma_x}^2.
\end{split}
\end{displaymath}

We consider now the case $N>N_0$. We take the extensions $\tilde{w}=\rho_T(w)$  and $\tilde{z}=\rho_T(z)$ defined in \eqref{defrho}, and we drop the tilde in the sequel. Arguing as in the proof of Proposition \ref{prop-ee}, we get
$$
  \cro{N^{-1}}^2  \cro{N}^{2\sigma} \widetilde{\mathcal{E}}_N(t) =  \cro{N^{-1}}^2\cro{N}^{2\sigma} \widetilde{\mathcal{E}}_N(0) - \widetilde{\mathcal{J}}_N + \widetilde{c_1}\widetilde{\mathcal{K}}_N + \widetilde{c_2} \widetilde{\mathcal{L}}_N
  $$
  with
  $$
  \widetilde{\mathcal{J}}_N = \cro{N^{-1}}^2 \cro{N}^{2\sigma}\int_{\R_t} P_N(zw)P_N\partial_xw
  $$
  and
  $$
  \widetilde{\mathcal{K}}_N = \cro{N^{-1}}^2 \cro{N}^{2\sigma}\int_0^t\frac{d}{dt}\widetilde{\mathcal{E}}_N^1(t')dt',\quad \widetilde{\mathcal{L}}_N =
   \cro{N^{-1}}^2 \cro{N}^{2\sigma}\int_0^t\frac{d}{dt}\widetilde{\mathcal{E}}_N^2(t')dt' .
  $$
  Proceeding as in the Section \ref{section-ee}, we split $\widetilde{\mathcal{J}}_N$ as $\widetilde{\mathcal{J}}_N^1+\widetilde{\mathcal{J}}_N^2+\widetilde{\mathcal{J}}_N^3$ with
  \begin{align*}
    \widetilde{\mathcal{J}}_N^1 &= N^{2\sigma} \int_{\R_t} \Pi_{\widetilde{\chi_1}}^2(\partial_xz_{\ll N}, w_{\sim N}) w_{\sim N},\\
    \widetilde{\mathcal{J}}_N^2 &= N^{2\sigma} \int_{\R_t} \Pi_{\widetilde{\chi_2}}^2(w_{\ll N}, z_{\sim N}) \partial_xw_{\sim N},\\
    \widetilde{\mathcal{J}}_N^3 &= N^{2\sigma} \sum_{N_1\gtrsim N} \int_{\R_t} P_N(z_{N_1}w_{\sim N_1})\partial_xw_N,
  \end{align*}
  where $\widetilde{\chi_1}=- \frac 12  \cro{N^{-1}}^2\chi_1 $ and $\widetilde{\chi_2}(\xi_1,\xi_2)= \cro{N^{-1}}^2\Bigl(\frac{\cro{N}}{N}\Bigr)^{2\sigma}\phi_N^2(\xi_1+\xi_2)$.

  \medskip

\noindent \textit{Estimate for $\widetilde{\mathcal{J}}_N^3$.} We infer from Proposition \ref{L2bilin} that
  \begin{align*}
    |\widetilde{\mathcal{J}}_N^3| &\lesssim \sum_{N_1\gtrsim N} N^{\sigma+\frac 12-\frac \alpha 4} N_1^{-s-\sigma+(1-\alpha)_+} \|z_{N_1}\|_{Y^s} \|w_{\sim N_1}\|_{Z^\sigma} \|w_N\|_{Z^\sigma}\\
    &\lesssim N^{(s_\alpha-s)_+} \|z\|_{Y^s_T} \|w\|_{Z^\sigma_T}^2 ,
  \end{align*}
  where in the last step we used that $ -s-\sigma+(1-\alpha)_+ <-(s-s_\alpha)_+<0 $ to sum over $ N_1$.
  Therefore we get
  \begin{equation}
    \sum_{N>N_0} |\widetilde{\mathcal{J}}_N^3| \lesssim  N_0^{(s_\alpha-s)_+} \|z\|_{Y^s} \|w\|_{Z^\sigma}^2.
  \end{equation}

  \medskip

\noindent \textit{Estimate for $-\widetilde{\mathcal{J}}_N^1+\widetilde{c_1}\widetilde{\mathcal{K}}_N$.} We deduce using equation \eqref{eqdiff} that
  \begin{align*}
  \widetilde{\mathcal{K}}_N &= -N^{2\sigma}\int_{\R^2_t} \left(\frac{\widetilde{\chi_1}}{\Omega_2}\right)(\xi_1,\xi_2) i\xi_1(\omega_{\alpha+1}(\xi_1)+\omega_{\alpha+1}(\xi_2)-\omega_{\alpha+1}(\xi)) \widehat{z_{\ll N}}(\xi_1) \widehat{w_{\sim N}}(\xi_2) \widehat{w_{\sim N}}(-\xi)\\
  &\quad + N^{2\sigma}\int_{\R^2_t} \left(\frac{\widetilde{\chi_1}}{\Omega_2}\right)(\xi_1,\xi_2) \xi_1 \widehat{P_{\ll N}\partial_x(u^2+v^2)}(\xi_1) \widehat{w_{\sim N}}(\xi_2) \widehat{w_{\sim N}}(-\xi)\\
  &\quad + N^{2\sigma}\int_{\R^2_t} \left(\frac{\widetilde{\chi_1}}{\Omega_2}\right)(\xi_1,\xi_2) \xi_1 \widehat{z_{\ll N}}(\xi_1) \widehat{P_{\sim N}\partial_x(zw)}(\xi_2) \widehat{w_{\sim N}}(-\xi)\\
  &\quad + N^{2\sigma}\int_{\R^2_t} \left(\frac{\widetilde{\chi_1}}{\Omega_2}\right)(\xi_1,\xi_2) \xi_1 \widehat{z_{\ll N}}(\xi_1) \widehat{w_{\sim N}}(\xi_2) \widehat{P_{\sim N}\partial_x(zw)}(-\xi)\\
  &:= -\widetilde{\mathcal{J}}_N^1 + \widetilde{\mathcal{K}}_N^1+\widetilde{\mathcal{K}}_N^2+\widetilde{\mathcal{K}}_N^3.
\end{align*}
We choose $\widetilde{c_1}=-1$ so that the first term on the right-hand side cancels out with $-\widetilde{\mathcal{J}}_N^1$ and it suffices to estimate $\widetilde{\mathcal{K}}_N^1+\widetilde{\mathcal{K}}_N^2+\widetilde{\mathcal{K}}_N^3$.

\medskip

\noindent \textit{Estimate for $\widetilde{\mathcal{K}}_N^1$.} The contribution of $\widetilde{\mathcal{K}}_N^1$ may be treated exactly as $K_N^1$ in the proof of Proposition \ref{prop-ee}. We obtain
\begin{equation}
  \sum_{N>N_0} |\widetilde{\mathcal{K}}_N^1| \lesssim (N_0^{-\alpha_+}+ N_0^{(s-s_\alpha)_+})(\|u\|_{Y^s}^2+\|v\|_{Y^s}^2) \|w\|_{Z^\sigma}^2.
\end{equation}

\medskip

\noindent \textit{Estimate for $\widetilde{\mathcal{K}}_N^2+\widetilde{\mathcal{K}}_N^3$.} We decompose $P_{\sim N}(zw)$ into dyadic pieces as follows:
\begin{equation}\label{prop-eed.3}
P_{\sim N}(zw)= z_{\ll N}w_{\sim N} + N^{-1}\Pi_{\widetilde{\chi}}^2(\partial_xz_{\ll N}, w_{\sim N}) + P_{\sim N}(z_{\sim N}w_{\lesssim N}) + P_{\sim N}(z_{\gg N}w_{\gg N}).
\end{equation}
As in the proof of Proposition \ref{prop-ee}, this leads to estimate $\sum_{j=1}^4\widetilde{\mathcal{K}}_N^{3j}$ where $\widetilde{\mathcal{K}}_N^{3j}$ denotes the contribution to $\widetilde{\mathcal{K}}_N^2+\widetilde{\mathcal{K}}_N^3$ of the jth term in the RHS of \eqref{prop-eed.3}.

\medskip

\noindent \textit{Estimate for $\widetilde{\mathcal{K}}_N^{31}$ and $\widetilde{\mathcal{K}}_N^{32}$.} Since in these terms, both occurrences of $w$ are localized at frequency $\sim N$, they may be estimated as $K_N^{31}$ and $K_N^{32}$ in the proof of Proposition \ref{prop-ee}. We infer that
\begin{equation}
  \sum_{N>N_0} (|\widetilde{\mathcal{K}}_N^{31}|+|\widetilde{\mathcal{K}}_N^{32}|) \lesssim    (N_0^{-\alpha_+}+ N_0^{(s_\alpha-s)_+})\|z\|_{Y^s}^2 \|w\|_{Z^\sigma}^2.
\end{equation}

\medskip

\noindent \textit{Estimate for $\widetilde{\mathcal{K}}_N^{33}$.}  It suffices to consider the contribution $\widetilde{\mathcal{K}}_N^{331}$ of $ \widetilde{\mathcal{K}}_N^2 $ to $\widetilde{\mathcal{K}}_N^{33}$ since the contribution
 to $ \widetilde{\mathcal{K}}_N^3 $ can be estimated in exactly the same way.
\begin{align*}
\widetilde{\mathcal{K}}_N^{331} &= N^{2\sigma} \int_{\R^2_t} \left(\frac{\widetilde{\chi_1}}{\Omega_2}\right)(\xi_1,\xi_2) \xi_1 \widehat{z_{\ll N}}(\xi_1) \widehat{P_{\sim N}\partial_x(z_{\sim N}w_{\lesssim N})}(\xi_2) \widehat{w_{\sim N}}(-\xi) \\
&= \sum_{N_1\ll N}\sum_{N_2\lesssim N} N^{2\sigma} (N_1N^\alpha)^{-1} N_1N \int_{\R_t} \Pi_{\widetilde{\chi}_{K^{331}}}^3(z_{N_1}, w_{N_2}, z_{\sim N}) w_{\sim N}
\end{align*}
where
$$
\widetilde{\chi}_{K^{331}}(\xi_1,\xi_2,\xi_3) = i\widetilde{\chi_1}(\xi,\xi_2+\xi_3) \frac{N_1N^\alpha}{\Omega_2(\xi_1,\xi_2+\xi_3)} \frac{\xi_1}{N_1} \frac{\xi_2+\xi_3}N \phi_{\sim N}(\xi_2+\xi_3)
$$
satisfies \eqref{pseudoprod.1}. Estimate \eqref{eL2trilin.2} gives
\begin{align*}
  |\widetilde{\mathcal{K}}_N^{331}| &\lesssim \sum_{N_1\ll N}\sum_{N_2\lesssim N} N^{(2\sigma+1-\alpha)_+} N_1^{(\frac12)_-} \cro{N_1}^{-\frac \alpha 4} N_2^{(\frac12)_-} \cro{N_2}^{-\frac \alpha 4}  \|z_{N_1}\|_{Y^0} \|w_{N_2}\|_{Z^0} \|z_{\sim N}\|_{Y^0} \|w_{\sim N}\|_{Z^0}\\
  &\lesssim N^{\sigma-s+(1-\alpha)_+}\sum_{N_1\ll N}\sum_{N_2\lesssim N} N_1^{(\frac12)_-} \cro{N_1}^{-\frac \alpha 4 -s} N_2^{(\frac12)_-} \cro{N_2}^{-\frac \alpha 4 -\sigma} \|z\|_{Y^s} \|w\|_{Z^\sigma} \|z_{\sim N}\|_{Y^s} \|w_{\sim N}\|_{Z^\sigma}.
\end{align*}
Since $\frac 12-\frac \alpha 4-s<0$ and $\frac 12-\frac \alpha 4-\sigma>0 $, this yields
\begin{equation}
  \sum_{N>N_0} |\widetilde{\mathcal{K}}_N^{331}| \lesssim \sum_{N>N_0} N^{(s_\alpha-s)_+} \|z\|_{Y^s_T}^2 \|w\|_{Z^\sigma_T}^2 \lesssim
   N_0^{(s_\alpha-s)_+}\|z\|_{Y^s}^2 \|w\|_{Z^\sigma}^2\; .
\end{equation}

\medskip

\noindent \textit{Estimate for $\widetilde{\mathcal{K}}_N^{34}$.} Again, we only estimate the contribution of
\begin{align*}
  \widetilde{\mathcal{K}}_N^{341} &=  N^{2\sigma} \int_{\R^2_t} \left(\frac{\widetilde{\chi_1}}{\Omega_2}\right)(\xi_1,\xi_2) \xi_1 \widehat{z_{\ll N}}(\xi_1) \widehat{P_{\sim N}\partial_x(z_{\gg N}w_{\gg N})}(\xi_2) \widehat{w_{\sim N}}(-\xi)\\
  &= \sum_{N_1\ll N} \sum_{N_2\gg N} N^{2\sigma} (N_1N^\alpha)^{-1} N_1N \int_{\R_t} \Pi_{\widetilde{\chi}_{K^{341}}}^3(z_{N_1}, z_{N_2}, w_{\sim N_2}) w_{\sim N}.
\end{align*}
It follows from estimate \eqref{eL2trilin.2} that
\begin{align*}
  |\widetilde{\mathcal{K}}_N^{341}| &\lesssim \sum_{N_1\ll N} \sum_{N_2\gg N} N^{(\sigma+s_\alpha)_-} N_2^{(-s-\sigma)_+} N_1^{(\frac12)_-} \cro{N_1}^{-\frac \alpha 4-s} \|z_{N_1}\|_{Y^s} \|z_{N_2}\|_{Y^s} \|w_{\sim N_2}\|_{Z^\sigma} \|w_{\sim N}\|_{Z^\sigma}\\
  &\lesssim  N^{(s_\alpha-s)_+} \|z\|_{Y^s}^2 \|w\|_{Z^\sigma}^2,
\end{align*}
where in the last step we used that $s+\sigma>0 $ and  $\frac 12-\frac \alpha 4-s<0$.
We conclude that
\begin{equation}
  \sum_{N>N_0} |\widetilde{\mathcal{K}}_N^{341}| \lesssim   N_0^{(s_\alpha-s)_+}\|z\|_{Y^s}^2 \|w\|_{Z^\sigma}^2.
\end{equation}

\medskip

\noindent \textit{Estimate for $-\widetilde{\mathcal{J}}_N^2+\widetilde{c_2}\widetilde{\mathcal{L}}_N$.} Using equation \eqref{eqdiff} we rewrite $\widetilde{\mathcal{L}}_N$ as
\begin{align*}
  \widetilde{\mathcal{L}}_N &= -N^{2\sigma}\int_{\R^2_t} \left(\frac{\widetilde{\chi_2}}{\Omega_2}\right)(\xi_1,\xi_2) i(\xi_1+\xi_2)(\omega_{\alpha+1}(\xi_1)+\omega_{\alpha+1}(\xi_2)-\omega_{\alpha+1}(\xi)) \widehat{w_{\ll N}}(\xi_1) \widehat{z_{\sim N}}(\xi_2) \widehat{w_{\sim N}}(-\xi)\\
  &\quad + N^{2\sigma}\int_{\R^2_t} \left(\frac{\widetilde{\chi_2}}{\Omega_2}\right)(\xi_1,\xi_2) (\xi_1+\xi_2) \widehat{P_{\ll N}\partial_x(zw)}(\xi_1) \widehat{z_{\sim N}}(\xi_2) \widehat{w_{\sim N}}(-\xi)\\
  &\quad + \frac 12N^{2\sigma}\int_{\R^2_t} \left(\frac{\widetilde{\chi_2}}{\Omega_2}\right)(\xi_1,\xi_2) (\xi_1+\xi_2) \widehat{w_{\ll N}}(\xi_1) \widehat{P_{\sim N}\partial_x(w^2)}(\xi_2) \widehat{w_{\sim N}}(-\xi)\\
  &\quad + \frac 12N^{2\sigma}\int_{\R^2_t} \left(\frac{\widetilde{\chi_2}}{\Omega_2}\right)(\xi_1,\xi_2) (\xi_1+\xi_2) \widehat{w_{\ll N}}(\xi_1) \widehat{P_{\sim N}\partial_x(z^2)}(\xi_2) \widehat{w_{\sim N}}(-\xi)\\
  &\quad + N^{2\sigma}\int_{\R^2_t} \left(\frac{\widetilde{\chi_2}}{\Omega_2}\right)(\xi_1,\xi_2) (\xi_1+\xi_2) \widehat{w_{\ll N}}(\xi_1) \widehat{z_{\sim N}}(\xi_2) \widehat{P_{\sim N}\partial_x(zw)}(-\xi)\\
  &:= -\widetilde{\mathcal{J}}_N^2 + \widetilde{\mathcal{L}}_N^1+\widetilde{\mathcal{L}}_N^2+\widetilde{\mathcal{L}}_N^3+\widetilde{\mathcal{L}}_N^4
\end{align*}
where we used that $z=u+v$ solves
$$
(\partial_t-L_{\alpha+1})z = \partial_x(u^2+v^2) = \frac 12(\partial_x(w^2)+\partial_x(z^2)) \text{ on } ]0,T[\; .
$$
Taking $\widetilde{c_2}=-1$ it remains to estimate $\sum_{j=1}^4 \widetilde{\mathcal{L}}_N^j$.

\medskip

\noindent \textit{Estimate for $\widetilde{\mathcal{L}}_N^1$.} We may rewrite this term as
$$
\widetilde{\mathcal{L}}_N^1 = \sum_{N_1\ll N} \sum_{N_2,N_3} N^{2\sigma} (N_1N^\alpha)^{-1} NN_1 \int_{\R_t} P_{N_1}(z_{N_2} w_{N_3}) \Pi_{\widetilde{\chi}_{L^1}}^2(z_{\sim N}, w_{\sim N})
$$
with
$$
\widetilde{\chi}_{L^1}(\xi_1,\xi_2) = i\widetilde{\chi}_2(-\xi_1-\xi_2, \xi_1) \frac{N_1N^\alpha}{\Omega_2(-\xi_1-\xi_2, \xi_1)} \frac{\xi_2}{N} \frac{\xi_1+\xi_2}{N_1}.
$$
The contribution $\widetilde{\mathcal{L}}_N^{11}$ of the region where $N_2\vee N_3\lesssim N$ is estimated thanks to \eqref{eL2trilin.2} by
\begin{align*}
  |\widetilde{\mathcal{L}}_N^{11}| &\lesssim \sum_{N_1\ll N} \sum_{N_2\vee N_3\lesssim N} N^{\sigma-s+1-\alpha} N_2^{(\frac12)_-} \cro{N_2}^{-\frac \alpha 4-s} N_3^{(\frac12)_-} \cro{N_3}^{-\frac \alpha 4-\sigma} N^{0_+} \\
  &\quad \times \|z_{N_2}\|_{Y^s} \|w_{N_3}\|_{Z^\sigma} \|z_{\sim N}\|_{Y^s} \|w_{\sim N}\|_{Z^\sigma}\\
  &\lesssim   N^{(s_\alpha-s)_+} \|z\|_{Y^s} \|w\|_{Z^\sigma} \|z_{\sim N}\|_{Y^s} \|w_{\sim N}\|_{Z^\sigma}
\end{align*}
where we used that $\frac{1}{2}-\frac \alpha 4 -s <0 $ and $ \frac{1}{2}-\frac \alpha 4 -\sigma>0 $. For the other contribution $\widetilde{\mathcal{L}}_N^{12}$, we must have $N_2\sim N_3$ and
by virtue of \eqref{eL2trilin.2} again, we deduce that
\begin{align*}
  |\widetilde{\mathcal{L}}_N^{12}| &\lesssim  \sum_{N_1\ll N} \sum_{N_2\sim N_3\gg N} N^{(\sigma-s+2-\frac 32 \alpha)_-} N_2^{-(s+\sigma)}  N_2^{0_+}
   \|z_{N_2}\|_{Y^s} \|w_{N_3}\|_{Z^\sigma} \|z_{\sim N}\|_{Y^s} \|w_{\sim N}\|_{Z^\sigma}\\
  &\lesssim   N^{2(s_\alpha-s)_+} \|z\|_{Y^s} \|w\|_{Z^\sigma} \|z_{\sim N}\|_{Y^s} \|w_{\sim N}\|_{Z^\sigma},
\end{align*}
where we used that $s+\sigma>0 $ and $ -s+1-\frac 34 \alpha\le s_\alpha-s$ for $\alpha\le 1$.
Therefore we infer that
 \begin{equation}
  \sum_{N>N_0} |\widetilde{\mathcal{L}}_N^1| \lesssim N_0^{(s_\alpha-s)_+}\|z\|_{Y^s}^2 \|w\|_{Z^\sigma}^2.
\end{equation}

\medskip

\noindent \textit{Estimate for $\widetilde{\mathcal{L}}_N^2$.} We need to bound
$$
\widetilde{\mathcal{L}}_N^2 = \sum_{N_1\ll N} \sum_{N_2,N_3} N^{2\sigma+2} (N_1N^\alpha)^{-1} \int_{\R_t} P_{\sim N}(w_{N_2}w_{N_3}) \Pi_{\widetilde{\chi}_{L^2}}^2(w_{N_1}, w_{\sim N})
$$
where
$$
\widetilde{\chi}_{L^2}(\xi_1,\xi_2) = \frac i2\widetilde{\chi}_2(\xi_1, -\xi_1-\xi_2) \frac{N_1N^\alpha}{\Omega_2(\xi_1, -\xi_1-\xi_2)} \frac{\xi_2}{N} \frac{\xi_1+\xi_2}{N_1}.
$$
We may always assume $N_2\le N_3$. The contribution $\widetilde{\mathcal{L}}_N^{21}$ of the sum over $N_2 \sim N_3\gg N$ is estimated thanks to Proposition \ref{L2trilin} by
\begin{align*}
  |\widetilde{\mathcal{L}}_N^{21}| &\lesssim \sum_{N_1\ll N} \sum_{N_2\gg N} N^{s-\frac 1 2 +\frac \alpha 4} N_1^{-(\frac12)_-} N_2^{-2s+(1-\alpha)_+} \|w_{N_1}\|_{Z^0} \|w_{N_2}\|_{Y^s} \|w_{\sim N_2}\|_{Y^s} \|w_{\sim N}\|_{Z^\sigma} \notag\\
  &\lesssim N^{(-s+\frac 1 2 -\frac 3 4 \alpha)_+} \|w\|_{Y^s}^2 \|w\|_{\overline{Z}^\sigma}^2.
\end{align*}
where in the first step we used that $\sigma<s-2+\frac32\alpha$ and in the last step we used that $ \sigma \ge - \frac 1 2 +\frac \alpha 4 >- \frac12_-$. We also used the weight $\cro{N_1^{-1}}$ of $\overline{Z}^{\sigma}$ to sum over $N_1 \le 1$. This leads to
\begin{equation} \label{prop-eed.7}
\sum_{N>N_0}   |\widetilde{\mathcal{L}}_N^{21}| \lesssim N_0^{(s_\alpha-s-1+\frac\alpha2)_+}  \|w\|_{Y^s}^2 \|w\|_{\overline{Z}^\sigma}^2
\lesssim N_0^{s_\alpha-s}  \|w\|_{Y^s}^2 \|w\|_{\overline{Z}^\sigma}^2.
\end{equation}

Similarly, we bound the contribution $\widetilde{\mathcal{L}}_N^{22}$ of the sum over $N_1\ll N_2$ and $N_3\sim N$ by
\begin{align*}
  |\widetilde{\mathcal{L}}_N^{22}| &\lesssim  \sum_{N_1\ll N_2\le N_3 \sim N} N^{2(\sigma-s +\frac 3 2 -\alpha)_+} N_1^{-(\frac12)_-} \langle N_2 \rangle^{-\frac 12-\frac \alpha 4} \|w_{N_1}\|_{Z^0} \|w_{N_2}\|_{Z^0} \|w_{\sim N}\|_{Y^s}^2 \notag\\
  &\lesssim N^{2(\sigma-s +2 -\frac{3}{2}\alpha)+(\alpha-1)_+} \|w\|_{\overline{Z}^\sigma}^2 \|w\|_{Y^s}^2,\label{prop-eed.8}
\end{align*}
where in the last step we used that $\sigma>-\frac 1 2 $. We also used the weight $\cro{N_1^{-1}}$ of $\overline{Z}^{\sigma}$ to sum over $N_1 \le 1$.
Setting $ \gamma=s -2 +\frac{3}{2}\alpha-\sigma >0 $, this leads to
\begin{equation}\label{prop-eed.8}
 \sum_{N>N_0} |\widetilde{\mathcal{L}}_N^{22}| \lesssim N_0^{-2\gamma+(\alpha-1)_+} \|w\|_{\overline{Z}^\sigma}^2 \|w\|_{Y^s}^2 \, .
\end{equation}
To deal with the last region $N_3\sim N$ and $N_2\lesssim N_1$, we use estimate  \eqref{eL2trilin.2} to get
\begin{align*}
  |\widetilde{\mathcal{L}}_N^{23}| &\lesssim \sum_{N_1\ll N} \sum_{N_2\lesssim N_1} N^{2(\sigma-s+1-\frac \alpha 2)_+} N_1^{-(\frac12)_-} \cro{N_1}^{-\frac \alpha 4-\sigma}  N_2^{(\frac12)_-} \cro{N_2}^{-\frac \alpha 4-\sigma} \|w_{N_1}\|_{Z^\sigma} \|w_{N_2}\|_{Z^\sigma} \|w_{\sim N}\|_{Y^s}^2 \notag\\
  &\lesssim  N^{2(\sigma-s+2-\frac{3}{2} \alpha)+2(\alpha-1)_+} (1+N^{ -\frac \alpha 2 -2\sigma})\|w\|_{\overline{Z}^\sigma}^2 \|w\|_{Y^s}^2, \label{prop-eed.9}
\end{align*}
where in the last step we used that $ \frac 1 2 -\frac \alpha 4 -\sigma >0 $ since $ \sigma<0 $ and that $ -\frac 1 2 -\frac \alpha 4 -\sigma <0$
 since $ \sigma>\frac \alpha 4 -\frac 1 2 $. It follows that
\begin{equation} \label{prop-eed.9}
\sum_{N>N_0}   |\widetilde{\mathcal{L}}_N^{23}| \lesssim ( N_0^{-2\gamma+2(\alpha-1)_+}+N_0^{(s_\alpha-s)_+} ) \|w\|_{Y^s_T}^2 \|w\|_{\overline{Z}^\sigma_T}^2
\end{equation}
and we deduce gathering \eqref{prop-eed.7}-\eqref{prop-eed.8}-\eqref{prop-eed.9} that
\begin{equation}
  \sum_{N>N_0} |\widetilde{\mathcal{L}}_N^2| \lesssim( N_0^{-2\gamma+(\alpha-1)_+}+N_0^{(s_\alpha-s)_+} ) (\|u\|_{Y^s}^2 + \|v\|_{Y^s}^2) \|w\|_{\overline{Z}^\sigma}^2.
\end{equation}
\medskip

\noindent \textit{Estimate for $\widetilde{\mathcal{L}}_N^3+\widetilde{\mathcal{L}}_N^4$.} Performing a dyadic decomposition for $P_{\sim N}(z^2)$ and $P_{\sim N}(zw)$, we get from \eqref{prop-ee.8} and \eqref{prop-eed.3} that
$$
\widetilde{\mathcal{L}}_N^3+\widetilde{\mathcal{L}}_N^4 = \sum_{i=1}^5 \widetilde{\mathcal{L}}_N^{4i}
$$
with
\begin{multline*}
\widetilde{\mathcal{L}}_N^{41} = N^{2\sigma} \int_{\R^2_t} \left(\frac{\widetilde{\chi}_2}{\Omega_2}\right) (\xi_1,\xi_2) (\xi_1+\xi_2) \widehat{w_{\ll N}}(\xi_1) \\ \times\left[ \widehat{\partial_x(z_{\ll N} z_{\sim N})}(\xi_2) \widehat{w_{\sim N}}(-\xi) + \widehat{z_{\sim N}}(\xi_2) \widehat{\partial_x(z_{\ll N}w_{\sim N})}(-\xi) \right],
\end{multline*}

\begin{multline*}
\widetilde{\mathcal{L}}_N^{42} = N^{2\sigma-1} \int_{\R^2_t} \left(\frac{\widetilde{\chi}_2}{\Omega_2}\right) (\xi_1,\xi_2) (\xi_1+\xi_2) \widehat{w_{\ll N}}(\xi_1) \big[ \F_x(\partial_x \Pi_{\widetilde{\chi}}^2(\partial_x z_{\ll N}, z))(\xi_2) \widehat{w_{\sim N}}(-\xi)\\
 + \widehat{z_{\sim N}}(\xi_2) \F_x(\partial_x \Pi_{\widetilde{\chi}}^2(\partial_x z_{\ll N}, w))(-\xi) \big],
\end{multline*}

$$
\widetilde{\mathcal{L}}_N^{43} = \frac 12 N^{2\sigma} \int_{\R^2_t} \left(\frac{\widetilde{\chi}_2}{\Omega_2}\right) (\xi_1,\xi_2) (\xi_1+\xi_2) \widehat{w_{\ll N}}(\xi_1) \F_x(\partial_xP_{\sim N}(z_{\gtrsim N}z_{\gtrsim N}))(\xi_2) \widehat{w_{\sim N}}(-\xi),
$$

$$
\widetilde{\mathcal{L}}_N^{44} = N^{2\sigma} \int_{\R^2_t} \left(\frac{\widetilde{\chi}_2}{\Omega_2}\right) (\xi_1,\xi_2) (\xi_1+\xi_2) \widehat{w_{\ll N}}(\xi_1) \widehat{z_{\sim N}}(\xi_2) \F_x(\partial_xP_{\sim N}(z_{\sim N}w_{\lesssim N}))(-\xi),
$$
and
$$
\widetilde{\mathcal{L}}_N^{45} = N^{2\sigma} \int_{\R^2_t} \left(\frac{\widetilde{\chi}_2}{\Omega_2}\right) (\xi_1,\xi_2) (\xi_1+\xi_2) \widehat{w_{\ll N}}(\xi_1) \widehat{z_{\sim N}}(\xi_2) \F_x(\partial_xP_{\sim N}(z_{\gg N}w_{\gg N}))(-\xi).
$$

\medskip

\noindent \textit{Estimate for $\widetilde{\mathcal{L}}_N^{41}$.} Arguing as for the term $K_N^{31}$ in the proof of Proposition \ref{prop-ee}, we obtain
$$
\widetilde{\mathcal{L}}_N^{41} = \sum_{N_1,N_2\ll N} N^{2\sigma+1}\frac{N_1\vee N_2}{N_1N^\alpha} G_t^3(w_{N_1}, z_{\sim N}, z_{N_2}, w_{\sim N}).
$$
The contribution $\widetilde{\mathcal{L}}_N^{411}$ of the sum over $N_2\lesssim N_1$ is bounded thanks to Proposition \ref{eL2trilin} by
\begin{align}
  |\widetilde{\mathcal{L}}_N^{411}| &\lesssim \sum_{N_1\ll N}\sum_{N_2\lesssim N_1} N^{(\sigma-s+1-\alpha)_+} N_1^{(\frac12)_-} \cro{N_1}^{-\frac \alpha 4-\sigma} N_2^{(\frac12)_-} \cro{N_2}^{-\frac \alpha 4-s}\\ & \quad \times \|w_{N_1}\|_{Z^\sigma} \|z_{\sim N}\|_{Y^s} \|z_{N_2}\|_{Y^s} \|w_{\sim N}\|_{Z^\sigma}\notag \\
  &\lesssim N^{(s_\alpha-s)_+} \|z\|_{Y^s}^2 \|w\|_{\overline{Z}^\sigma}^2. \label{prop-eed.11}
\end{align}
Using Proposition \ref{L2trilin}, the other contribution $\widetilde{\mathcal{L}}_N^{412}$ is estimated by
\begin{align}
  |\widetilde{\mathcal{L}}_N^{412}| &\lesssim \sum_{N_1\ll N_2\ll N} N^{-(\frac\alpha2)_+} N_1^{-(\frac12)_-} \cro{N_2}^{\frac 12-\frac \alpha 4-s} \|w_{N_1}\|_{Z^0} \|z_{\sim N}\|_{Y^s} \|z_{N_2}\|_{Y^s} \|w_{\sim N}\|_{Z^\sigma}\notag \\
  &\lesssim N^{-(\frac\alpha2)_+} \|z\|_{Y^s}^2 \|w\|_{\overline{Z}^\sigma}^2. \label{prop-eed.12}
\end{align}
since $ s>\frac 1 2 - \frac \alpha 4$, $\sigma-s+2<\frac32\alpha$ and where we also used the weight $\cro{N_1^{-1}}$ of $\overline{Z}^{\sigma}$ to sum over $N_1 \le 1$.
Combining estimates \eqref{prop-eed.11}-\eqref{prop-eed.12} we infer that
\begin{equation}
  \sum_{N>N_0} |\widetilde{\mathcal{L}}_N^{41}| \lesssim  (N_0^{(s_\alpha-s)_+} + N_0^{-(\frac\alpha2)_+})\|z\|_{Y^s}^2  \|w\|_{\overline{Z}^\sigma}^2.
\end{equation}

\medskip

\noindent \textit{Estimate for $\widetilde{\mathcal{L}}_N^{42}$.} Noticing that
$$
\widetilde{\mathcal{L}}_N^{42} = \sum_{N_1,N_2\ll N} N^{2\sigma+1}\frac{N_2}{N_1N^\alpha} G_t^3(w_{N_1}, z_{\sim N}, z_{N_2}, w_{\sim N}),
$$
it is clear that we may follow the same lines as the estimate for $\widetilde{\mathcal{L}}_N^{41}$ to prove
\begin{equation}
  \sum_{N>N_0} |\widetilde{\mathcal{L}}_N^{42}| \lesssim   (N_0^{(s_\alpha-s)_+} + N_0^{-(\frac\alpha2)_+})\|z\|_{Y^s}^2  \|w\|_{\overline{Z}^\sigma}^2.
\end{equation}

\medskip

\noindent \textit{Estimate for $\widetilde{\mathcal{L}}_N^{43}$, $\widetilde{\mathcal{L}}_N^{44}$ and $\widetilde{\mathcal{L}}_N^{45}$.} It is not too hard to check that $\widetilde{\mathcal{L}}_N^{43}$ and $\widetilde{\mathcal{L}}_N^{45}$ may be estimated as $\widetilde{\mathcal{L}}_N^{21}$ above, whereas we can deal with $\widetilde{\mathcal{L}}_N^{44}$ by following the bounds on $
\widetilde{\mathcal{L}}_N^{22}$ and $ \widetilde{\mathcal{L}}_N^{23}$. Thus we get
\begin{equation}
  \sum_{N>N_0} (|\widetilde{\mathcal{L}}_N^{43}|+|\widetilde{\mathcal{L}}_N^{44}|+|\widetilde{\mathcal{L}}_N^{45}|) \lesssim  ( N_0^{-2\gamma_+(\alpha-1)_+}+N_0^{(s_\alpha-s)_+} )\|z\|_{Y^s}^2  \|w\|_{\overline{Z}^\sigma}^2.
\end{equation}

This concludes the proof of Proposition \ref{prop-eed}.

\end{proof}
\section{Proof of Theorem \ref{maintheo}}
Let us fix $ 0<\alpha\le 1 $.
\subsection{Lipschitz bound and uniqueness}
Let $ s>s_\alpha $,  $ 0<T\le 1 $ and assume that  $ u\in Y^s_T $ and $v\in Y^s_{T}$ are two solutions to \eqref{dpB} on $ ]0,T[ $ associated with  initial data $ u_0, v_0 \in H^s(\R) $ such that
 $ u_0-v_0 \in \overline{L}^2(\R) $.
We fix  $ -\frac 1 2 +\frac \alpha 4< \sigma<\min (0,s-2+\frac{3}{2} \alpha) $ and  set $w=u-v$.
 It is clear that $ w(0)=w_0 \in \overline{H}^\sigma $ and the continuous embedding from $ Y^s_{T} $ into $ Z^\sigma_{T} $ ensures that $w\in Z^\sigma	_{T} $.
 Now, from Duhamel formula we have
 $$
P_{\le 1} w(t) = P_{\le 1} U_\alpha(t) w_0 + \int_0^t U_\alpha(t-t') P_{\le 1}\partial_x (u^2-v^2)(t') \, dt'
$$
and thus ,
\begin{align*}
\|P_{\le 1} w\|_{L^2_{T} \overline{L}^\infty_x} & \lesssim  \| P_{\le 1} w\|_{L^\infty_{T} \overline{L}^2_x}
 \lesssim  \|w_0\|_{\overline{L}^2_x} + \sum_{N\le 1} N \cro{N^{-1}} N^{\frac 1 2}
\|P_{N}(u^2-v^2)\|_{L^\infty_{T} L^1_x}\\
&  \lesssim  \|w_0\|_{\overline{L}^2_x}+\|u\|_{L^\infty_{T} L^2_x}^2+\|v\|_{L^\infty_{T} L^2_x}^2 \; .
\end{align*}
Moreover, classical linear estimates in the context of Bourgain's space (cf. \cite{Bo1}, \cite{G}) lead to
$$
\|P_{\le 1} w \|_{\overline{X}_{T}^{\sigma-1,1}}\lesssim  \|w_0\|_{\overline{L}^2} +\|P_{\le 1 } (u^2-v^2)\|_{L^2_{{T}x}} \le
 \|w_0\|_{\overline{L}^2} +\|u\|_{L^\infty_{T} L^2_x}^2+\|v\|_{L^\infty_{T} L^2_x}^2
$$
These estimates combined with \eqref{coco2} and the fact that $ w\in Z^\sigma_{T} $, ensure that $ w\in \overline{Z}^\sigma_{T} $.

Combining Corollary \ref{coro2}, Lemma \ref{lem-EsTt} and Proposition \ref{prop-eed},  we obtain that, for any  $ N_0 \gg   (1+\|z\|_{L^\infty_{T} H^s_x})^\frac{2}{\alpha} $,
\begin{displaymath}
\begin{split}
\|w\|_{L^\infty_{T} \overline{H}^\sigma_x}^2 & \lesssim  \|w_0\|_{\overline{H}^\sigma_x}^2+({T} N_0^{\frac32}+N_0^{-(\frac\alpha2)_+}+ N_0^{(s_\alpha-s)_+}+N_0^{-2\gamma+(\alpha-1)_+})\\
 & \quad \quad \times (1+\|u\|_{Y^s_{T}}^2+\|v\|_{Y^s_{T}}^2)^3 \|w\|_{L^\infty_{T} \overline{H}^\sigma_{x}}^2
\end{split}
\end{displaymath}
 where $ \gamma=s -2 +\frac{3}{2}\alpha-\sigma >0 $.
Taking
$
N_0 \gg (1+\|u\|_{Y^s_{T}}^2+\|v\|_{Y^s_{T}}^2)^{\frac 3 \delta}$ with
$$
\delta=\min\big\{(\frac{\alpha}{2})_-, (s-s_\alpha)_-,(1-\alpha)_- +2\gamma\big\} >0 \, .
$$
This forces
\begin{equation}\label{lip}
\|w\|_{L^\infty_{T'} \overline{H}^\sigma_x} \lesssim \|w_0\|_{\overline{H}^\sigma_x}
\end{equation}
 for $0< T'\lesssim \min\Bigl\{ (1+\|u\|_{Y^s_T}^2+\|v\|_{Y^s_{T}}^2)^{-\frac{9}{2\delta}},T\Bigr\}$.

Therefore, taking $ u_0-v_0=0 $, we obtain that $ u\equiv v $ on $]0,T'[ $. Noticing, that equation \eqref{fKdV} ensures that $ u_t ,v_t \in L^\infty(0,T;H^{s-2}(\mathbb R))$
 and thus $ u,v \in C([0,T];L^2(\R)) $, it follows that
 $ v(T')=u(T')$.
Repeating this argument a finite number of times we extend the uniqueness  result on $ ]0,T[$.
\subsection{A priori estimates on smooth solutions}
  According to \cite{S} (see also \cite{bonasmith} to get the continuity of the flow-map) for any $ u_0\in H^\theta(\R) $, with $ \theta\ge 3$, there exists a positive time $T=T(\|u_0\|_{H^3}) $ and a unique solution $ u\in C([0,T];H^\theta(\R)) $  to \eqref{dpB}  emanating from $ u_0 $. Moreover, for any fixed $ R>0$,  the map $ u_0 \mapsto u$ is continuous from the ball of $ H^\theta(\R) $ of radius $ R $ centered at the origin into $ C([0,T]; H^\theta(\R)) $.
  \vspace*{4mm}

Let $u_0\in H^\infty(\R )$. From the above result $ u_0 $ gives rise to a solution $ u\in C([0,T^*[;H^\infty(\R)) $ to \eqref{dpB} with $ T^*\ge T(\|u_0\|_{H^3})$ and
\begin{equation}\label{exp}
\lim_{t\nearrow T^*}\|u(t)\|_{H^3}=+\infty   \quad \;\mbox{if} \;  T^*<+\infty.
\end{equation}
Let $ 0<T<T^*$. Since $u\in C([0,T];H^\infty(\R)) $ is a solution to \eqref{dpB}, we must have $ u_t \in L^\infty(0,T; H^\infty(\R)) $ and thus it is easy to check that
 $u\in Y^\theta_T $ for any $ \theta\in\R$ and
 \begin{equation}\label{gj}
 \lim_{T\searrow 0} \|u\|_{Y^\theta_T} =\|u_0\|_{H^\theta} \; .
 \end{equation}
  In the sequel, $ \kappa>0 $ and $ C_0>1 $ are  the constants appearing in Corollary \ref{coro1}.

 We claim that   there exist $ A_0>0$, $0<\beta_0\ll 1 $  such that $ T^*\ge A_0   (1+\|u_0\|_{H^{s'}})^{-\frac 1 \beta_0} $  and, for any $ s_\alpha<s'\le 3 $,
 \begin{equation}\label{cla1}
 \|u\|_{Y^{s'}_T} \le 2^2 C_0 \|u_0\|_{H^{s'}} \quad\mbox{ with } T=A_0   (1+\|u_0\|_{H^{s'}})^{-\frac 1 \beta_0}\; .
 \end{equation}
 Indeed, fixing $ s_\alpha<s'\le 3 $, it follows from \eqref{gj}  that
 $$
 \Lambda_{s'}=\big\{T\in ]0, T^*[ \, : \, \|u\|_{Y^{s'}_T}^2\le 2^4 C_0^2 \|u_0\|_{H^{s'}}^2\big\}
 $$
 is a non empty interval of $\R_+^*$.  Let us set $ T_0 =\sup \Lambda_{s'} $.  We proceed by contradiction, assuming that  $ T_0<A_0(1+\|u_0\|_{H^{s'}})^{-\frac 1 \beta_0}$ since otherwise we are done. Note that by continuity  $$
 \|u\|_{Y^{s'}_{T_0}} ^2\le 2^4 C_0^2 \|u_0\|_{H^{s'}}^2 \; .
 $$
  According to   Corollary \ref{coro1}, Lemma \ref{lem-EsT} and Proposition \ref{prop-ee}, there exist $ C_1,C_2\gg1 $ and $ 0<\varepsilon_0\ll 1 $ such that for any  $s>s_{\alpha}$, $ N_0\ge C_1  (1+\|u_0\|_{H^{s_\alpha}})^{\frac 1 \alpha}$  and  any $0<T<\min\{\varepsilon_0 \|u_0\|_{H^{s'}}^{-\frac 1 \kappa}, T_0\} $, it holds
 \begin{equation}\label{dod}
 \|u\|_{Y^{s}_T}^2 \le 4 C_0^2 \|u_0\|_{H^{s}}^2+ C_2 (T N_0^{\frac 3 2} +N_0^{(s_\alpha-s')_+}+N_0^{-\alpha_+})
 (1+\|u_0\|_{H^{s'}})^2 \|u\|_{Y^{s}_T}^2
 \end{equation}
 We  take $ A_0\le \varepsilon_0 $ and $ \beta_0\le \kappa $ so that $ \min\{\varepsilon_0 \|u_0\|_{H^{s'}}^{-\frac 1 \kappa}, T_0\}=T_0 $ and thus, by continuity,
  \eqref{dod} is satisfied with $ T=T_0$.
  Now, applying  \eqref{dod} with
  $s=3$, $ N_0 = [8C_2 (1+\|u_0\|_{H^{s'}})^2]^{{\frac 1 \delta}+} $, where $ \delta=\min\{\alpha, {s'}-s_\alpha \}  $, and
  $ T=\min\{T_0, (8C_2 N_0^{\frac 3 2})^{-1}\} $, we  get
   \begin{equation}\label{dod2}
  \|u\|_{Y^3_{T}}^2 \le 8 C_0^2 \|u_0\|_{H^{3}}^2\;.
  \end{equation}
  Therefore, taking $ A_0\le  \varepsilon_0 $ and  $ \beta_0\le \kappa $ small enough so that
  $$
  (8C_2 N_0^{\frac 3 2})^{-1}=\Bigl[ 8 C_2 \Bigl( 8C_2(1+\|u_0\|_{H^{s'}})^2 \Bigr)^{(\frac{3}{2\delta})_+}\Bigr]^{-1}
  > A_0(1+\|u_0\|_{H^{s'}})^{-\frac 1 \beta_0}
  $$
  we obtain that \eqref{dod2} is satisfied with $ T=T_0$.
  In view of \eqref{exp}, this forces $ T^*>T_0$. Now taking  $ s=s' $ and proceeding in the same way we get
  $$
  \|u\|_{Y^{s'}_{T_0}}^2 \le 8 C_0^2 \|u_0\|_{H^{s'}}^2\;.
  $$
  But since $ T^*>T_0 $, by continuity this ensures that $  \|u\|_{Y^{s'}_{T}}^2 \le 2^4 C_0^2  \|u_0\|_{H^{s'}}^2$ for some $ T_0<T<T^* $ which contradicts the definition of $ T_0$ .
 This concludes the proof of \eqref{cla1}.\\
 Note that Lemma \ref{lem-EsT} and Corollary \ref{coro-ee}  then ensure that for any $ N_0\ge C_1  (1+\|u_0\|_{H^{s_\alpha}})^{\frac 1 \alpha}$, it holds
 \begin{equation}\label{dod3}
 \|P_{\ge N_0} u \|_{L^\infty_T H^{s'}_x}^2 \lesssim \|P_{\ge N_0} u_0\|_{H^{s'}}^2+ (N_0^{-\alpha_+}+N_0^{(s_\alpha-s)_+}) (1+\|u_0\|_{H^{s'}})^4
 \end{equation}
 where $ T>0 $ is defined as in \eqref{cla1}.
 \subsection{Local existence in $ H^s(\R) $, $s>s_\alpha $.}
 Now let  us fix $s>s_\alpha $  and $ u_0\in H^s(\R) $. We set $ u_{0,n}=P_{\le n} u_0 $ and we denote by $ u_n \in C([0,T_n^{\star});H^\infty(\R)) $ the
 solutions to \eqref{dpB} emanating from $u_{0,n} $. Setting
 \begin{equation}\label{defT}
  T=A_0   (1+\|u_0\|_{H^{s}})^{-\frac 1 \beta_0}\, ,
  \end{equation}
   it follows from   \eqref{cla1}  that  for any $ n\in \N^* $, $T_n^{\star} \ge T$ and
\begin{equation} \label{bob}
\|u_n\|_{Y^s_T} \lesssim   \|u_0\|_{H^s}  \; .
\end{equation}
Let  $ -\frac 1 2 +\frac \alpha 4< \sigma<\min (0,s-2+\frac{3}{2} \alpha) $. For $ n\ge m\ge 1 $, clearly $ u_{0,n}-u_{0,m} \in \overline{L}^2(\R) $ and thus \eqref{lip} ensures that
$$
\| u_n-u_m\|_{L^\infty_{T^{''}} H^{\sigma}_x} \lesssim \|u_{0,n}-u_{0,m}\|_{\overline{H}^\sigma} \lesssim \|P_{\ge m} u_0 \|_{H^s} \; ,
$$
where $0< T^{''}=T^{''}(\|u_0\|_{H^{s'}}) \le T$. This last inequality combined with \eqref{dod3} ensure that
\begin{eqnarray}
\| u_n-u_m\|_{L^\infty_{T^{''}} H^{s}_x}^2
& \lesssim  &  N_0^{2(s-\sigma)}\|P_{< N_0}(u_n-u_m)\|_{L^{\infty}_{T^{''}}H^{\sigma}_x}^2 \\ & &+\|P_{\ge N_0}u_n\|_{L^{\infty}_{T^{''}}H^{s}_x}^2+\|P_{\ge N_0}u_m\|_{L^{\infty}_{T^{''}}H^{s}_x}^2 \nonumber\\
& \lesssim  &  N_0^{2(s-\sigma)}\|P_{\ge m}  u_0\|_{H^{s}}^2+ \|P_{\ge N_0} u_0 \|_{H^{s}}^2\nonumber\\
 & &+ (N_0^{-\alpha+}+N_0^{(s_\alpha-s)+}) (1+\|u_0\|_{H^{s}})^4  \label{bbb}
\end{eqnarray}
for any $  N_0\ge C_1  (1+\|u_0\|_{H^{s_\alpha}})^{\frac 1 \alpha}$.
This proves that $ \{u_n\} $ is a Cauchy sequence in
 $ C([0,T^{''}]; H^s(\R)) $ and thus converges to some $ u $ in this space. It is then not hard to check that $ u\in Y^s_{T^{''}} $ and is a solution to
  \eqref{dpB} emanating from $ u_0$. By the uniqueness result, this is the only one. Repeating this argument a finite number of times we obtain that actually
   $\{ u_n\} $ converges to $ u $ in $  C([0,T]; H^s(\R)) $ with $ T $ defined in \eqref{defT}.
\subsection{Continuity of the solution-map}
  Finally,  to prove the continuity with respect to initial data, we take a sequence $ \{u_0^j\} \subset B_{H^s}(0,2\|u_0\|_{H^s}) $ that converges to $u_0 $ in $ H^s(\R) $.
  We denote by respectively $ u^j$ and $ u^j_n $  the associated solutions to  \eqref{dpB} emanating from respectively $ u_0^j $ and $ P_{\le n} u_0^j $. Noticing that
  $$
  \lim_{m\to +\infty} \sup_{j\in\N} \|P_{\ge m}(u^j_0)\|_{H^s}=0 \,,
  $$
  we infer from \eqref{bbb} that
  $$
  \lim_{n\to +\infty}\sup_{j\in\N} \|u^j-u^j_n\|_{L^\infty_{T^{''}} H^{\sigma}_x}=0 \; .
  $$
  with $  T^{''}=T^{''}(\|u_0\|_{H^{s}})>0$. From
  $$
  \|u^j-u \|_{L_T^\infty H^s_x} \le \|u^j-u^j_n \|_{L_T^\infty H^s_x}  + \| u^j_n-u_n  \|_{L_T^\infty H^s_x} + \| u_n -u  \|_{L_T^\infty H^s_x}\,
  $$
 and  the continuity with respect to initial data in $ H^3(\R) $ (note that $ P_{\le n} u_0 $ and $ P_{\le n} u_0 $ belong to $ H^\infty(\R)$), it follows that $ u^j \to u $ in $ C([0,T^{''}]; H^s(\R)) $. Iterating this process a finite number of times we obtain that  $ u^j \to u $ in $ C([0,T]; H^s(\R)) $
  with $T $ defined in \eqref{defT} which  completes the proof of Theorem \ref{maintheo}.

\bibliographystyle{amsplain}

\end{document}